\documentclass{amsart}
\usepackage[margin=0.95in]{geometry}
\usepackage{comment}
\usepackage{mathrsfs}
\usepackage{amsmath}
\usepackage{amsthm}
\usepackage{amsfonts}
\usepackage{amssymb}
\usepackage[shortlabels]{enumitem}
\usepackage{graphicx}
\usepackage{url}
\usepackage{tikz}
\usepackage{palatino}
\usepackage[linktocpage=true]{hyperref}
\usepackage{tikz-cd}
\hypersetup{
    colorlinks=true,
    linkcolor=blue,
}
\usepackage{setspace}
\usepackage{relsize}
\usepackage{bm}
\usepackage{tikz}
\usepackage{wrapfig}
\usetikzlibrary{snakes}
\usepackage{pgfplots}
\usepackage{amsmath,amssymb,amsfonts}
\usepgfplotslibrary{fillbetween}
\usetikzlibrary{patterns}
\usetikzlibrary {patterns,patterns.meta}
\usepackage{capt-of}
\usepackage{subcaption}
\newtheorem{theorem}{Theorem}[section]
\newtheorem{proposition}[theorem]{Proposition}
\newtheorem{lemma}[theorem]{Lemma}
\newtheorem{corollary}[theorem]{Corollary}

\theoremstyle{definition}
\newtheorem{definition}[theorem]{Definition}
\newtheorem{example}[theorem]{Example}

\theoremstyle{remark}
\newtheorem{remark}[theorem]{Remark}
\numberwithin{equation}{section}
\theoremstyle{remark}
\newtheorem{claim}[theorem]{Claim}
\usepackage{cleveref}

\newcommand{\R}{\mathbb{R}}
\newcommand{\C}{\mathbb{C}}

\newcommand{\Z}{\mathbb{Z}}

\newcommand{\module}{PR_\kappa}

\makeatletter

\renewcommand\part{%
   \if@noskipsec \leavevmode \fi
   \par
   \addvspace{4ex}%
   \@afterindentfalse
   \secdef\@part\@spart}

\def\@part[#1]#2{%
    \ifnum \c@secnumdepth >\m@ne
      \refstepcounter{part}%
      \addcontentsline{toc}{part}{\thepart\hspace{1em}#1}%
    \else
      \addcontentsline{toc}{part}{#1}%
    \fi
    {\parindent \z@ \raggedright
     \interlinepenalty \@M
     \normalfont
     \ifnum \c@secnumdepth >\m@ne
       \Large\bfseries \partname\nobreakspace\thepart
       \par\nobreak
     \fi
     \huge \bfseries #2%
     \par}%
    \nobreak
    \vskip 3ex
    \@afterheading}
\def\@spart#1{%
    {\parindent \z@ \raggedright
     \interlinepenalty \@M
     \normalfont
     \huge \bfseries #1\par}%
     \nobreak
     \vskip 3ex
     \@afterheading}
\makeatother

\usepackage[backend=bibtex,style=alphabetic,maxalphanames=4,maxnames=4]{biblatex}

\renewbibmacro{in:}{}

\DeclareDelimFormat[bib,biblist]{nametitledelim}{\addcomma\space}

\DeclareFieldFormat*{title}{\mkbibitalic{#1}\addcomma}
\DeclareFieldFormat*{journaltitle}{#1}
\DeclareFieldFormat*{volume}{\mkbibbold{#1}}
\DeclareFieldFormat{pages}{#1}
\DeclareFieldFormat[misc]{date}{preprint {#1}}
\DeclareFieldFormat{mr}{%
  MR\addcolon\space
  \ifhyperref
    {\href{http://www.ams.org/mathscinet-getitem?mr=MR#1}{\nolinkurl{#1}}}
    {\nolinkurl{#1}}}
\AtEveryBibitem{
  \clearfield{url}
 \clearfield{number}
  \clearfield{doi}
  \clearfield{issn}
  \clearfield{isbn}
  \clearfield{eprintclass}
}
\AtEveryBibitem{\ifentrytype{book}{\clearfield{pages}}{}}

\bibliography{PolyRep}

\begin{document}

\title{A Floer-theoretic interpretation of the polynomial representation of the double affine Hecke algebra}
\author{Eilon Reisin-Tzur}
\address{University of California, Los Angeles, Los Angeles, CA 90095}
\email{ereisint@math.ucla.edu}
\thanks{This work partially supported by NSF grant DMS-2003483.}

\maketitle

\begin{abstract}
We construct an isomorphism between the wrapped higher-dimensional Heegaard Floer homology of $\kappa$-tuples of cotangent fibers and $\kappa$-tuples of conormal bundles of homotopically nontrivial simple closed curves in $T^*\Sigma$ with a certain braid skein group, where $\Sigma$ is a closed oriented surface of genus $> 0$ and $\kappa$ is a positive integer. Moreover, we show this produces a (right) module over the surface Hecke algebra associated to $\Sigma$. This module structure is shown to be equivalent to the polynomial representation of DAHA in the case where $\Sigma=T^2$ and the cotangent fibers and conormal bundles of curves are both parallel copies. 

\end{abstract}

\tableofcontents

\section{Introduction}

% \textcolor{blue}{Write this at the end. Some motivation for why we actually care about this.. finding a geometric realization of this algebraic structure. These DAHA modules and polynomial representations only exist on the algebra side but we can find a relatively simple geometric interpretation for this. It might also help compute the HDHF algebra for things like these, since the algebra is generally more wieldy than making sure you have counted all of the curves. 
% This should also have some relation to a Topological quantum field theory where to each three manifold we associate to it a Hecke algebra and this corresponds to the solid torus. Not so sure about that last statement....}\\

Higher-dimensional Heegaard Floer homology (HDHF) was developed by Colin, Honda, and Tian in \cite{colin2020applications} to analyze symplectic fillability questions in higher-dimensional contact topology. As its name suggests, it is very closely related to Heegaard Floer homology introduced by Ozsv\'{a}th and Szab\'{o} in \cite{ozsvath2004holomorphic} to study closed oriented 3-manifolds. HDHF models the Fukaya category of the Hilbert scheme of points on a Liouville domain and has been used to produce an invariant of links in $S^3$. In \cite{honda2022higher}, Honda, Tian, and Yuan constructed isomorphisms between the wrapped HDHF of cotangent fibers of cotangent bundles of closed oriented surfaces $\Sigma$ with positive genus to Hecke algebras ($H_\kappa(\Sigma)$) associated with the $\Sigma$. In particular, they showed that the wrapped HDHF of cotangent fibers of $T^*T^2$ is isomorphic to the double affine Hecke algebra (DAHA) $\ddot{H}_\kappa$ introduced by Cherednik in \cite{cherednik} for his proof of Macdonald's conjectures. Using results by Morton and Samuelson in \cite{morton2021dahas}, Honda, Tian, and Yuan were able to provide a symplectic geometry (Floer-theoretic) interpretation of DAHA and various other Hecke algebras. 

The goal of this paper is to build on \cite{honda2022higher} to provide a symplectic geometry interpretation of the polynomial representation of DAHA, closely related to Cherednik's basic representation. 

Let $\kappa$ be a positive integer and fix $\kappa$ distinct points $q_1, \dots, q_\kappa \in \Sigma$. Given the isomorphism between the wrapped HDHF $CW(\sqcup_{i=1}^\kappa T^*_{q_i}\Sigma)_c$ of cotangent fibers of $T^*\Sigma$  with an additional parameter $c$ and the surface Hecke algebra tensor product $H_\kappa(\Sigma) \otimes \Z[[\hbar]]$, there exists a functor from the HDHF Fukaya category (with parameter $c$) to the category of (right) $\hbar$-deformed $H_\kappa(\Sigma)$-modules which sends mutually disjoint Lagrangians $L_1,\dots, L_\kappa$ to Hom$(\sqcup_{i=1}^\kappa T^*_{q_i}\Sigma, \sqcup_{i=1}^\kappa L_i)$. This paper aims to shed more light on this functor by giving it a more explicit topological interpretation. 

Abbondandolo, Portaluri, and Schwarz proved in \cite{abbondandolo2008floer} that Floer homology with conormal boundary conditions is isomorphic to the singular homology of the natural path space associated to the boundary conditions. In the case of cotangent fibers, Abouzaid improved this to an $A_\infty$-equivalence on the chain level in \cite{abouzaid2012wrapped}.

We discuss the generalization of these results to HDHF. More precisely, we define the wrapped HDHF cochain complex between $\kappa$-tuples of mutually disjoint conormal bundles. Restricting to the manifold $T^*\Sigma$ and conormal bundles $\sqcup_{i=1}^\kappa T^*_{q_i}\Sigma, \sqcup_{i=1}^\kappa N^* \alpha_i$, where $\alpha_1,\dots, \alpha_\kappa$ is a mutually disjoint collection of homotopically nontrivial simple closed curves in $\Sigma$, we find that the wrapped HDHF is concentrated in degree zero. 

Our generalization of the path space in \cite{abbondandolo2008floer} is the path space of the unordered configuration space $\text{UConf}_\kappa(\Sigma)$ of $\kappa$ points on $\Sigma$ satisfying boundary conditions to ensure the paths go between our conormal bundles.  

Following \cite{abouzaid2012wrapped} and \cite{honda2022higher}, we define an evaluation map 
$$\mathcal{E}: CW(\sqcup_{i=1}^\kappa \phi^1_{H_V}(T^*_{q_i} \Sigma),\sqcup_{i=1}^\kappa N^* \alpha_i) \longrightarrow C_0(\Omega(\text{UConf}_\kappa(\Sigma),\bm{q}, \bm{\alpha})) \otimes \Z[[\hbar]],$$
where $\bm{q} = \{q_1,\dots,q_\kappa \} \in \text{UConf}_\kappa(\Sigma)$, $\bm{\alpha} = \alpha_1 \times \cdots \times \alpha_\kappa$, and $\Omega(\text{UConf}_\kappa(\Sigma),\bm{q}, \bm{\alpha})$ is the space of paths in $\text{UConf}_\kappa(\Sigma)$ starting at $\bm{q}$ and ending in $\bm{\alpha}$. Here we are viewing an element $(x_1,\dots,x_\kappa)\in \bm{\alpha}$ as an unordered tuple; this is possible since the $\alpha_1,\dots,\alpha_\kappa$ are mutually disjoint.  

Taking homotopy classes of paths in $C_0(\Omega(\text{UConf}_\kappa(\Sigma),\bm{q}, \bm{\alpha}))$ and quotienting by the HOMFLY skein relation produces a map 
$$\mathcal{F}:CW(\sqcup_{i=1}^\kappa T_{q_i}^*\Sigma,\sqcup_{i=1}^\kappa N^* \alpha_i) \longrightarrow BSk_\kappa(\Sigma,\bm{q}, \bm{\alpha}),$$
\noindent where $BSk_\kappa(\Sigma,\bm{q}, \bm{\alpha})$ is a variant of the braid skein algebra introduced by Morton and Samuelson in \cite{morton2021dahas}. Informally, $BSk_\kappa(\Sigma,\bm{q}, \bm{\alpha})$ consists of homotopy classes of braids which start at $\{q_1,\dots, q_\kappa\}$ and end in $\alpha_1 \times \cdots \times \alpha_\kappa$, modulo the HOMFLY skein relation. Morton and Samuelson also define the braid skein algebra on a punctured surface, leading to our variant $BSk_\kappa(\Sigma,\bm{q},\bm{\alpha}, \ast)$, defined in Section \ref{sec:parameterc}. The puncture $\ast \in \Sigma$ gives rise to a marked point relation which in turn is interpreted as a $c$-deformed homotopy relation on our braids. Adding this marked point into our formulation, we get a map 
$$\mathcal{F}:CW(\sqcup_{i=1}^\kappa T_{q_i}^*\Sigma,\sqcup_{i=1}^\kappa N^* \alpha_i)_c \longrightarrow BSk_\kappa(\Sigma,\bm{q},\bm{\alpha}, \ast).$$

The first main result of this paper is Theorem \ref{thm:Fiso1}, proved in Section \ref{sec:Fiso proof} :

\begin{theorem}
\label{thm:Fiso1}
    $\mathcal{F}$ is an isomorphism.
\end{theorem}

With this in mind, we show that the action of $CW(\sqcup_{i=1}^\kappa T^*_{q_i} \Sigma)_c$ on $CW(\sqcup_{i=1}^\kappa T^*_{q_i}\Sigma,\sqcup_{i=1}^\kappa N^* \alpha_i)_c$ agrees with the action of the braid skein algebra $BSk_\kappa(\Sigma,\bm{q},\ast)$ on $BSk_\kappa(\Sigma,\bm{q},\bm{\alpha},\ast)$. 

\begin{lemma}
\label{lemma:presentation}
Let $\alpha_1, \dots, \alpha_\kappa$ be parallel copies of the meridian on $T^2$. Then
$$BSk_\kappa(T^2,\bm{q},\bm{\alpha}, \ast) \simeq (\Z[a_1^{\pm 1}, \dots, a_{\kappa}^{\pm 1}] \otimes \Z[S_{\kappa}])  \otimes \Z[c^{\pm 1}] \otimes \Z[[\hbar]].$$
\end{lemma}

For the configuration of points $q_i$ and curves $\alpha_i$ as in Lemma \ref{lemma:presentation}, we introduce an extra homological variable $d$ which keeps track of sliding the ends of the braids past each other on the $\alpha_i$; see Definition \ref{def:d-deformed}. 
The next main result is obtained after setting $d=s$ and relating $(\Z[a_1^{\pm 1}, \dots, a_{\kappa}^{\pm 1}] \otimes \Z[S_{\kappa}])  \otimes \Z[c^{\pm 1}] \otimes \Z[[\hbar]] \otimes \Z[d^{\pm 1}]$ to $\Z[[s]][c^{\pm 1}][X_1, \dots, X_\kappa]$ by averaging over the permutation components:

\begin{theorem}
\label{thm:mainaction}
    The action of $CW(\sqcup_{i=1}^\kappa T^*_{q_i} T^2)_c$ on $CW(\sqcup_{i=1}^\kappa T^*_{q_i}T^2,\sqcup_{i=1}^\kappa N^* \alpha_i)_{c,d}$ agrees with the polynomial representation of $\ddot{H}_\kappa$ on $\Z[[s]][c^{\pm 1}][X_1, \dots, X_\kappa]$ after setting $d=s$ and averaging the permutation components of 
    
    \noindent$CW(\sqcup_{i=1}^\kappa T^*_{q_i}T^2,\sqcup_{i=1}^\kappa N^* \alpha_i)_{c,d}$.
\end{theorem}

\noindent\textit{Organization}: In Section 2, we give a brief review of HDHF, wrapped HDHF, and conormal boundary conditions. In Section 3, we introduce the results of Abbondandolo and Schwarz before generalizing to the case $\kappa \geq 1$. We give a summary of \cite{honda2022higher} in Section 4, along with the proof of Theorem \ref{thm:Fiso1}. Section 5 discusses the action of $CW(\sqcup_{i=1}^\kappa T^*_{q_i}\Sigma)_c$ on $CW(\sqcup_{i=1}^\kappa T^*_{q_i}\Sigma, \sqcup_{i=1}^\kappa N^*\alpha_i)_c$ in the general case, before specializing to $T^2$ and the curves $\alpha_i$ in Section 6, where we prove Lemma \ref{lemma:presentation} and Theorem \ref{thm:mainaction}. \\

\textbf{Acknowledgements.} I would like to thank Ko Honda for nearly countless discussions and guidance through this project. I would also like to thank Tianyu Yuan for helpful discussions and suggestions and Peter Samuelson for discussions which motivated this study.

\section{Review of HDHF, wrapped HDHF, and conormal boundary conditions}

We give a brief summary of HDHF in Section 2.1 before defining the specific wrapped version of interest in Section 2.2.

\subsection{Review of HDHF} 
\hfill\\
\label{sec:HDHFreview}We refer the reader to \cite{colin2020applications} for more details regarding HDHF.

\begin{definition}
    Let $(X,\alpha)$ be a $2n$-dimensional completed Liouville domain and let $\omega = d\alpha$ be the exact symplectic form on $X$. The objects of the $A_\infty$-category $\mathcal{F}_\kappa (X)$ are $\kappa$-tuples of disjoint exact Lagrangians. The morphisms Hom$_{\mathcal{F}_\kappa (X)}(L_0,L_1) = CF(L_0,L_1)$ between two such objects $L_i = L_{i1} \sqcup \dots \sqcup L_{i\kappa}$, $i=0,1$, with mutually transverse components is the free abelian group generated by $\kappa$-tuples of intersections where each component is used exactly once. That is, the generators are $\bm{y} = \{ y_1, \dots, y_\kappa \}$ where $y_j \in L_{0j} \cap L_{1\sigma(j)}$ for some permutation $\sigma$ of $\{ 1, \dots , \kappa \}$. The coefficient ring is $\Z [[\hbar]]$ and the $A_\infty$-operations $\mu^m$ will be defined by (\ref{highermaps}).
\end{definition}

Similarly to the cylindrical reformulation of Heegaard Floer homology by Lipschitz \cite{lipshitz2006cylindrical}, we introduce an extra direction to keep track of points in the symmetric product of $X$. Let $D$ be the unit disk in $\C$ and $D_m = D - \{p_0, \dots , p_m \}$ be the disk with $m$ boundary punctures arranged counterclockwise. Let $\partial_i D_m$ be the boundary component from $p_i$ to $p_{i+1}$, with $\partial_m D_m$ going from $p_m$ to $p_0$.  We choose representatives of the moduli space of $D_m$ modulo automorphisms and label these representatives $D_m$, for lack of a better name. The {\em $A_\infty$-base direction} $D_m$ is shown in Figure \ref{fig:base}. 

\begin{figure}[h]
    \centering
    \includegraphics[width=0.25\textwidth]{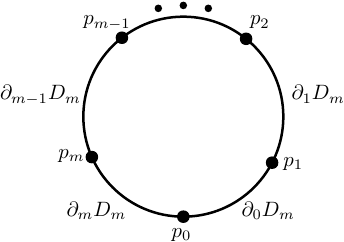}
    \caption{The $A_\infty$ base}
    \label{fig:base}
\end{figure}

Consider the manifold $\tilde{X}=(D_m \times X, \omega_m + \omega)$, where $\omega_m$ is an area form on $D_m$ which restricts to $ds_i \wedge dt_i$ on each strip-like end $e_i$ near $p_i$. We take $s_0 \to -\infty$ as we approach the {\em negative end} $p_0$ and $s_i \to + \infty$ for the other punctures refered to as the {\em positive ends}. Given $m+1$ objects $L_0, \dots, L_m$, we let $\tilde{L}_i = \partial_iD_m \times L_i$. We denote by $\pi_X : D_m \times X \to X$ the projection onto X and by $\pi_{D_m}: D_m \times X \to D_m$ the symplectic fibration of the base. 

There is a smooth assignment $D_m \mapsto J_{D_m}$ of almost complex structures $J_{D_m}$ that are close to a split almost complex structure $j_m \times J_X$ and which project holomorphically onto $D_m$. For more details, the reader is referred to \cite{colin2020applications} or \cite{honda2022higher}.

\begin{remark}
We call an assignment of almost complex structures $\textit{sufficiently generic}$ if all of the moduli spaces under consideration are transversely cut out. 
\end{remark}

Let $\mathcal{M}(\bm{y}_1,\dots,\bm{y}_m,\bm{y}_0)$ be the moduli space of maps 
$$u: (\dot{F},j)\longrightarrow (D_m \times X, J_{D_m}),$$ where $(F,j)$ is a compact Riemann surface with boundary, $\bm{p}_0$, $\dots$, $\bm{p}_m$ are disjoint $\kappa$-tuples of boundary punctures of $F$, and $\dot{F} = F \setminus \cup_i \bm{p}_i$, such that $u$ satisfies: 
\begin{equation}
    \begin{cases}
    $$du \circ j = J_{D_m} \circ du$$;\\
    $$\text{each component of } \partial \dot{F} \text{ is mapped to a unique }\tilde{L}_{ij}$$;\\
    $$\pi_X \circ u \text{ approaches }\bm{y}_i \text{ as } s_i \to +\infty \text{ for }i=1,\dots,m$$;\\
    $$\pi_x \circ u \text{ tends to }\bm{y}_0 \text{ as }s_0 \to -\infty$$;\\
    $$\pi_{D_m} \circ u \text{ is a }\kappa\text{-fold branched cover of }D_m.$$
    \end{cases}
\end{equation}\\
Letting the boundary punctures $\bm{p}_0$, $\dots$, $\bm{p}_m$ vary, the $A_\infty$ composition map $$\mu^m: CF(L_{m-1},L_m) \otimes \cdots \otimes CF(L_0,L_1) \longrightarrow CF(L_0,L_m)$$ is then defined as 
\begin{equation}
\label{highermaps}
    \mu^m(\bm{y}_1,\dots,\bm{y}_m)= \sum_{\bm{y}_0, \chi \leq \kappa} \# \mathcal{M}^{ind=0,\chi}(\bm{y}_1,\dots,\bm{y}_m,\bm{y}_0)\cdot \hbar^{\kappa - \chi}\cdot \bm{y}_0,\
\end{equation}
where $\chi$ is the Euler characteristic of $\dot{F}$ and $\#$ is the signed count of the moduli space. 

\begin{theorem}
\label{indextheorem}
    The Fredholm index of $\mathcal{M}^{\chi}(\bm{y}_1,\dots,\bm{y}_m,\bm{y}_0)$ with a varying complex structure on $D_m$ is $$ind(u) = (n-2) \chi + \mu + 2\kappa - m\kappa n + m - 2,$$
    where $\mu$ is the Maslov index of u. We refer the reader to \cite{colin2020applications} for details on a similar formula. 
\end{theorem}

\begin{remark}
    In the case where $2c_1(TX) = 0$ and the Maslov classes of the involved Lagrangians vanish, we have that $|\hbar| = 2-n$. 
\end{remark}

\subsection{Wrapped Floer theory of conormal bundles}
\hfill\\
We give a quick summary of wrapped Floer homology of conormal bundles; more details can be found in \cite{abbondandolo2008floer}.

% Let $M$ be a closed manifold (in our case this will be $T^2$), and let $Q$ be a closed submanifold of $M \times M$. We will eventually choose $Q$ to be the product $T^*_qT^2 \times N^* \alpha$, where $N^*\alpha$ is the \hyperlink{conorm}{conormal bundle} of an embedded closed curve $\alpha$. 
Let $M$ be a closed manifold and $T^*M$ be its cotangent bundle. 
Denoting the elements of $T^*M$ as pairs $(q,p) \in M \times T_q^*M$, let $\omega= dp \wedge dq$ be the the standard symplectic form on $T^*M$. Furthermore, let $\eta$ be the Liouville vector field satisfying $\mathcal{L}_\eta \omega = \omega$. Given a time-dependent Hamiltonian $H: [0,1] \times T^*M \to \R$, let $X_{H}$ be the unique vector field such that $-dH(Y) = \omega(X_H,Y)$ for any vector field $Y$ on $T^*M$. We look for solutions $x:[0,1] \to T^*M$ of the non-local boundary value Hamiltonian equation 
\begin{equation}
\label{eq:Hameq}
    x'(t) = X_H(t,x(t)), \ x(0)\in L_0, \ x(1) \in L_1,
\end{equation}

\noindent where $L_0$ and $L_1$ are Lagrangian submanifolds of $T^*M$, not of $M$.

With this in mind, we consider smooth Hamiltonians $H$ on $[0,1] \times T^* M$ such that: 
\begin{enumerate}
   
    \item[\hypertarget{H0}{(H0)}] every solution $x$ of the non-local boundary value Hamiltonian problem is nondegenerate
   
    \item[(H1)] there exist  $h_0 > 0$ and $h_1\geq 0$ such that
$$DH(t,q,p)[\eta] - H(t,q,p) \geq h_0 |p|^2 - h_1,$$ for every $(t,q,p) \in [0,1] \times T^* M$ 
    \label{H2}
    \item[(H2)] there exists an $h_2 \geq 0$ such that 
    $$|\nabla_q H(t,q,p)| \leq h_2(1+|p|^2), \ |\nabla_p H(t,q,p) \leq h_2(1+|p|)|,$$ 
    for every $(t,q,p) \in [0,1] \times T^* M$, where $\nabla_q$ and $\nabla_p$ denote the horizontal and vertical components of the gradient, respectively.
\end{enumerate}
 Condition (H0) holds for a generic choice of $H$ in the space that we are considering. Conditions (H1) and (H2) ensure that $H$ grows quadratically on the fibers of $T^*M$ and is radially convex for large $|p|$. 
 % There is a unique vector field $X_H$ such that $-dH(Y) = \omega(X_H,Y)$ for any vector field Y. Let $(t,x) \mapsto \phi^H(t,x)$ be the flow associated to the Hamiltonian vector field $X_H$. \\

% Let $N^*Q$ denote the \hypertarget{conorm}{\textit{conormal bundle}} of $Q$ in $M\times M$, i.e. the set of covectors $(q,p)$ in $T^*(M \times M)$ such that $q \in Q$ and $p \in T^*_q(M \times M)$ vanishes identically on $T_q Q$. The conormal bundle $N^*Q$ is a Lagrangian submanifold of $T^*(M \times M)$ and so is a natural candidate for investigation. 

% In this paper, we are interested in the higher-dimensional Heegaard Floer groups between two conormal bundles of $M$, whose product is a conormal bundle in $M\times M$. Specifically, we have the conormal bundle of a point, $T^*_q T^2$, and the conormal bundle of an embedded closed curve $\alpha \in T^2$, $N^* \alpha$. If $q$ is not on $\alpha$, then these bundles do not intersect and there are no elements in our group under the current definition. We will introduce one way to deal with this now and then show another generalization in the form of wrapped HDHF in the following section. \\

% We now give a quick overview of a version of wrapped Floer homology associated to a Hamiltonian H. For more details, the reader is directed to [AS08]. 
We reframe our non-local boundary value Hamiltonian problem \eqref{eq:Hameq} as is done in \cite{abbondandolo2008floer}. 

Let $Q$ be a submanifold of $M \times M$ which is either compact or has cylindrical ends, meaning it is the union of a compact submanifold and submanifolds of the form $K \times [0,\infty)$, where $K$ is a Legendrian. Let $N^*Q$ denote the \hypertarget{conorm}{\textit{conormal bundle}} of $Q$, i.e. the set of covectors $(q,p)$ in $T^*(M \times M)$ such that $p \in T^*_q(M \times M)$ vanishes identically on $T_q Q$. The conormal bundle $N^*Q$ is a Lagrangian submanifold of $T^*(M \times M)$ and so is a natural candidate for investigation. 

\begin{remark}
    Given two Lagrangians $L_0$ and $L_1$ in $T^* M$, it is not necessarily true that $L_0 \times L_1$ is the conormal bundle of some submanifold $Q \subset M\times M$. On the other hand, not every $N^*Q$ can be written as $L_0 \times L_1$ for Lagrangians $L_0,L_1 \in T^*M$. Therefore, the problem we are tackling neither contains nor is contained in the original formulation above. For the purposes of this paper, we will be interested in Lagrangians $L_i$ of $T^*M$ which are conormal bundles of submanifolds $Q_i$ in $M$. Under these conditions, this new formulation is equivalent to the first one since $N^*(Q_1 \times Q_2) = L_1 \times L_2$. Specifically, we will be interested in the case where $M$ is a closed oriented surface of positive genus and $Q = \{\text{pt}\}\times \{\text{pt}\}$ or $Q= \{\text{pt}\} \times \alpha$ for a homotopically nontrivial simple closed curve $\alpha$. 
\end{remark}

Let $H$ be a time-dependent Hamiltonian on $T^*M$ satisfying \hyperlink{H0}{(H0), (H1), and (H2)} and consider the set of solutions, $\mathcal{P}^Q(H)$, to the Hamiltonian equation with \textit{conormal boundary conditions}. 

That is, $\mathcal{P}^Q(H)$ is the set of $x:[0,1] \to T^*M$ satisfying $$x'(t) = X_H(t,x(t)),$$ subject to the boundary conditions $$(x(0),\mathcal{C}x(1)) \in N^* Q,$$ where $\mathcal{C}$ is the anti-symplectic involution $T^* M \to T^* M$, $(q,p) \mapsto (q,-p)$. Note that if $Q = Q_1 \times Q_2 \in M \times M$, then $\mathcal{P}^Q(H)$ is the set of trajectories from $Q_1$ to $Q_2$ along our Hamiltonian vector field $X_H$. 

Let $\theta$ be the Liouville one-form on $T^* M$ and consider the Hamiltonian action functional given by 
$$\mathbb{A}_H(x) := \int x^*(\theta - Hdt).$$

The first variation of $\mathbb{A}_H(x)$ on the space of free paths is 
$$d \mathbb{A}_H(x)[\xi] = \int_0^1 \omega(\xi,x'(t)-X_H(t,x))dt + \theta(x(1))[\xi(1)] - \theta(x(0))[\xi(0)],$$
where $\xi$ is a section of $x^*(TT^*M)$. Since $\theta$ vanishes on the conormal bundle of every submanifold of $M$, the extremal curves of $\mathbb{A}_H(x)$ are precisely the elements of $\mathcal{P}^Q(H)$. This is the reason we set this up with conormal bundles.

The conditions \hyperlink{H0}{(H0), (H1), and (H2)} that we imposed on our Hamiltonian $H$ imply that the set of solutions $x\in \mathcal{P}^Q(H)$ such that $\mathbb{A}_H(x)\leq A$ is finite. 

Given a smoothly time-dependent $\omega$-compatible almost complex structure $J$ on $T^*M$, we consider the Floer equation 
\begin{equation}
\label{Floereq}
\partial_s u + J(t,u) (\partial_t u - X_H(t,u)) = 0,
\end{equation}

\noindent where $u: \R \times [0,1] \to T^*M$. 

Let $x^-, x^+ \in \mathcal{P}^Q(H)$ and denote by $\mathcal{M}(x^-,x^+)$ the set of all solutions of the Floer equation \eqref{Floereq} with the non-local boundary condition such that 
$$\lim_{s \to \pm \infty} u(s,t) = x^{\pm} (t), \ \forall t \in [0,1].$$
Then one can show that we have an energy identity 
\begin{equation}
    E(u) := \mathbb{A}_H(x^-) - \mathbb{A}_H(x^+).
\end{equation}

Furthermore, $\mathcal{M}(x^-,x^+)$ is empty whenever $\mathbb{A}_H(x^-) \leq \mathbb{A}_H(x^+)$ and $x^- \neq x^+$ and it consists of only the element $u(s,t) = x(t)$ when $x^-=x^+=x$. 

By perturbing the almost complex structure $J$, we can give $\mathcal{M}(x^-,x^+)$ a smooth structure and its dimension is the difference of Maslov indices,
$$\text{dim} \mathcal{M}(x^-,x^+) = \mu^Q(x^-) - \mu^Q(x^+).$$

It follows that when $\mu^Q(x^-) - \mu^Q(x^+)=1$, we get an oriented one-dimensional manifold. Moreover, we have a free $\R$-action given by translation of the $s$ variable and so we arrive at a compact zero-dimensional manifold $\mathcal{M}(x^-,x^+)/ \R$. Let $\epsilon([u]) \in \{-1,1\}$ be $+1$ if the $\R$-action is orientation-preserving on the component of $\mathcal{M}(x^-,x^+)$ containing $u$, and $-1$ otherwise. Define 
$$n_F(x^-,x^+):= \sum_{[u] \in \mathcal{M}(x^-,x^+)/\R} \epsilon([u]),$$
and denote by $F_k^Q(H)$ the free Abelian group generated by the elements $x\in \mathcal{P}^Q(H)$ with Maslov index $k$. 
The boundary morphism $$\partial_k:F_k^Q(H) \longrightarrow F_{k-1}^Q(H)$$ 
is defined by 
$$\partial_kx^- := \sum_{x^+\in \mathcal{P}^Q(H), \mu^Q(x^+) = k-1} n_F(x^-,x^+)x^+.$$

Since the set of elements with an upper bound on the action is finite, the above sum is finite. It can be shown that $\partial_{k-1} \circ \partial_k = 0$, and so $\{F_*^Q(H),\partial_* \}$ is a complex of free Abelian groups, called the Floer complex of $(T^* M, Q, H, J)$. The Floer homology is then defined as usual from the complex. 

As usual, different choices of the Hamiltonian $H$ produce chain homotopy equivalent complexes (provided the Hamiltonians are close enough).

\subsection{Wrapped HDHF}
\hfill\\
We offer a different but equivalent formulation of the wrapped Floer homology defined in the previous subsection and extend it to wrapped HDHF. 

Let $(M,g)$ be a compact Riemmanian manifold of dimension $n$ with the induced norm $|\cdot|$ on $T^*M$. Choose a time-dependent Hamiltonian $H_V: [0,1] \times T^*M \to \R$:
$$H_V(t,q,p)=\frac{1}{2}|p|^2 + V(t,q),$$

\noindent where $q \in M, p\in T^*_q M$ and $V$ is a perturbation term with small $W^{1,2}$-norm. This Hamiltonian satisfies the conditions mentioned in the previous subsection. \\
Taking the standard symplectic form $\omega = dq \wedge dp$ on $T^*M$, let $X_{H_V}$ be the Hamiltonian vector field and $\phi^t_{H_V}$ be the time-$t$ flow of $X_{H_V}$.

Let $L_0, L_1$ be Lagrangian submanifolds of $T^*M$ with cylindrical ends. The time-1 flow, $\phi^1_{H_V}(L_0)$, is again a Lagrangian submanifold of $T^*M$ with cylindrical ends. We define the \textit{wrapped Floer chain complex} $CW(L_0,L_1)$ of $L_0$ and $L_1$ to be the Floer chain complex $CF(\phi^1_{H_V}(L_0),L_1)$ of $\phi^1_{H_V}(L_0)$ and $L_1$. 

\begin{proposition}
    Let $Q = Q_1 \times Q_2 \in M\times M$ be a product of submanifolds whose conormal bundles have cylindrical ends. Then given a Hamiltonian $H_V$ satisfying conditions as before, 
    $$CW(N^*Q_1,N^*Q_2) = F^Q(H_V).$$
\end{proposition}

\begin{proof}
    In both cases, the generators are time-1 Hamiltonian chords from $N^*Q_1$ to $N^*Q_2$. Moreover, the differentials for both count the same 0-dimensional moduli space of pseudoholomorphic curves between generators.
\end{proof}

We now generalize the definition of wrapped Floer theory to wrapped HDHF. 

Consider disjoint $\kappa$-tuples of Lagrangians $L_i = \sqcup_{i=j}^\kappa L_{ij}$ for $i = 1, 2$. We can ensure that the Hamiltonian chords between all of the Lagrangians involved are non-degenerate by choosing $g$ and $V$ generically.
\begin{definition}
    The \textit{wrapped higher dimensional Heegaard Floer chain complex} is given by $$CW(L_1,L_2):= CF(\sqcup_{i=1}^\kappa \phi^1_{H_V}(L_{1i}),\sqcup_{i=1}^\kappa L_{2i}).$$

\end{definition}

In the case of $L_1 = L_2$, we let $CW(L_1) := CW(L_1,L_1) = CF(\sqcup_{i=1}^\kappa \phi^1_{H_V}(L_{1i}),\sqcup_{i=1}^\kappa L_{1i})$. The $A_\infty$-operation
$$\mu^m: CW(L_1) \otimes \cdots \otimes CW(L_1) \longrightarrow CW(L_1)$$
\noindent does not immediately follow from the $A_\infty$-operations in the non-wrapped HDHF. Writing the Lagrangians out carefully, we see that the map is actually 
$$\mu^m: CF(\sqcup_i \phi^1_{H_V}(L_{1i}),\sqcup_i L_{1i})  \otimes \cdots \otimes CF(\sqcup_i \phi^m_{H_V}(L_{1i}),\sqcup_i \phi^{m-1}_{H_V}(L_{1i}))\to CF(\sqcup_i \phi^m_{H_V}(L_{1i}),\sqcup_iL_{1i}),$$
where $i=1,\cdots,\kappa$. The subtlety is that while  $CF(\sqcup_{i=1}^\kappa \phi^l_{H_V}(L_{1i}),\sqcup_{i=1}^\kappa \phi^{l-1}(L_{1i}))$ is naturally isomorphic to $CF(\sqcup_{i=1}^\kappa \phi^1_{H_V}(L_{1i}),\sqcup_{i=1}^\kappa L_{1i})$ for all $l \in \Z$, it is not the case for the chain complex $CF(\sqcup_{i=1}^\kappa \phi^m_{H_V}(L_{1i}),\sqcup_{i=1}^\kappa L_{1i})$. Luckily, this is resolved by a rescaling argument outlined in \cite{honda2022higher}, following \cite{abouzaid10}.

Taking this one step further, given $\kappa$-tuples of Lagrangians $L_1$ and $L_2$ with cylindrical ends, we can give $CW(L_1, L_2)$ the structure of a (right) $A_\infty$-module over $CW(L_1)$. Using a similar rescaling argument for $d\geq 2$ we can define the $A_\infty$-maps
$$\mu^d: CW(L_1, L_2) \otimes CW(L_1) \otimes \cdots \otimes CW(L_1) \longrightarrow CW(L_1, L_2),$$
\noindent giving us a (right) $A_\infty$-module structure.

\subsection{Wrapped HDHF example}
\hfill\\
\label{modelcalc}We perform a model calculation with $\kappa = 2$. Identify $M=T^2$ with $S^1 \times S^1 = \R / \Z \times \R / \Z$. Fix points $q_1 = (\frac{1}{6},\frac{1}{6})$, $q_2 = (\frac{2}{6},\frac{2}{6})$ and curves $\alpha_1= \{\frac{4}{6} \} \times S^1$ and $\alpha_2 = \{\frac{5}{6}\} \times S^1$. Let $L_1 = \sqcup_{i=1}^2 T^* _{q_i} T^2$ and $L_2 = \sqcup_{i=1}^2 N^* \alpha_i$. Note that $N^* \alpha_i = \alpha_i \times (\R \times \{0\}) \subset T^* T^2$.

 The perturbation term $V(t,q)$ can be chosen arbitrarily small and we will disregard it for the sake of this model computation. Then taking $H_V(t,q,p)=\frac{1}{2}|p|^2$, we have that $X_H = -p_1 \partial_{q_1} - p_2 \partial_{q_2}$ and the flow is given by $\phi^t_{H}(q_1,q_2,p_1,p_2) = (q_1 - p_1 t, q_2 - p_2 t, p_1, p_2)$, where $(q_i,p_i)$ are viewed as coordinates on $T^* S^1$. Since the cotangent fibers are based at $(\frac{i}{6},\frac{i}{6})$, the time-1 flow is given by $\phi^1_H(\frac{i}{6},\frac{i}{6},p_1,p_2) = (\frac{i}{6} - p_1, \frac{i}{6} - p_2, p_1,p_2)$. 

We are interested in intersections of $\phi^1_H(T^*_{q_i}T^2)$ with $N^* \alpha_j$, so we want solutions to $(\frac{i}{6} - p_1, \frac{i}{6} - p_2, p_1,p_2)$ = $(\frac{1}{2}+\frac{j}{6},a,p,0)$ for some $a \in \R / \Z,\  p \in \R$. Then $p_2$ must be 0. (Note that with a perturbation term, this is no longer necessarily true since there will be some flow in the fiber direction, but $p_2$ would be very small.) Continuing with this example, we see that $\frac{i}{6} - p_1 =  \frac{1}{2}+\frac{j}{6}$, and so $p_1 = -\frac{1}{2} + \frac{i-j}{6}$. Given one such $p_1$, we see that adding or subtracting 1 from $p_1$ gives another intersection. Let $\pi: T^* T^2 \to T^2$, $(q,p) \mapsto q$, be the projection onto the zero section. Then adding (subtracting) 1 along the fiber direction $p_1$ simply wraps clockwise (counterclockwise) once more around the $S^1 \times \{\frac{i}{6}\}$ direction on the torus before intersecting $\alpha_j$. The wrapping, along with some Hamiltonian chords, is shown in Figure \ref{fig:generators} for $\kappa=1$.
\begin{figure}[h]
    \centering
    \includegraphics[width=0.75\textwidth]{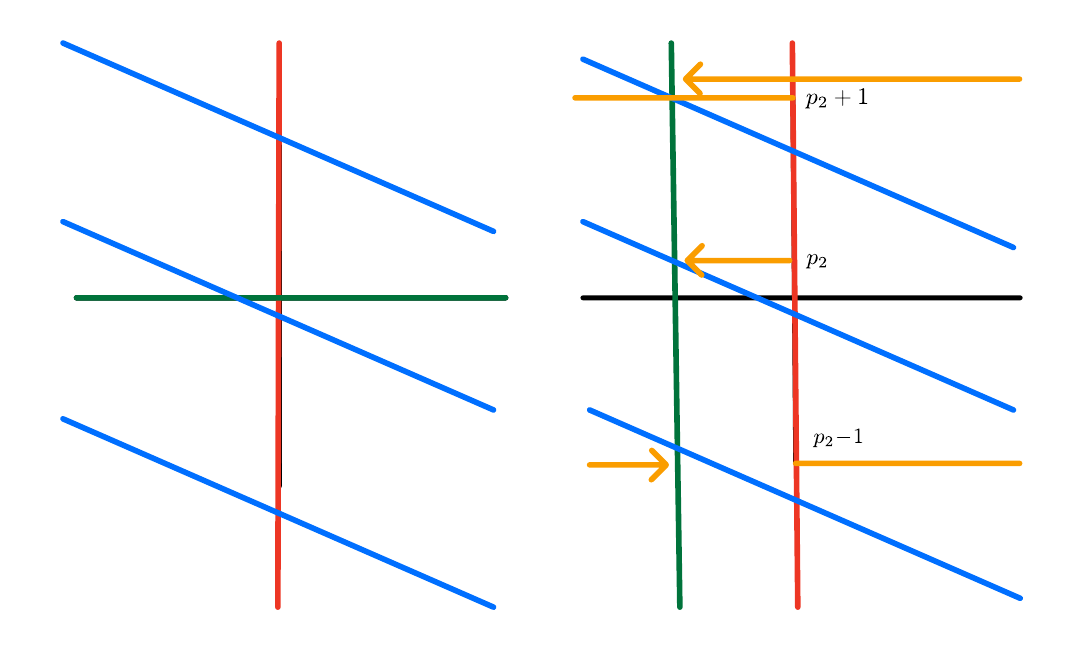}
    \caption{The case where $\kappa$=1. Here we view $T^*T^2$ as two copies of $T^* S^1$. The curve in red is $T^*_qT^2$, $\phi^1_{H_V}(T_q^*T^2)$ is in blue, and $N^* \alpha$ is in green. The arrows in orange represent Hamiltonian chords going from $T^*_qT^2$ to $N^* \alpha$. }
    \label{fig:generators}
\end{figure}

\begin{claim}
    The generators of $CF(\sqcup_{i=1}^2 \phi^1_{H_V}(T^*_{q_i} T^2),\sqcup_{i=1}^2 N^* \alpha_i)$ are elements of the form $(a_1^{n_1}a_2^{n_2},\sigma)$ \noindent where $\sigma \in S_2$ indicates intersections of $\phi^1_{H_V}(T^*_{q_i} T^2)$ and $N^* \alpha_{\sigma(i)}$ for $i=1,2$ and the exponent of $a_i$ indicates the number of times the intersection point, viewed as a Hamiltonian chord, wraps around the $(-1,0)$-direction when projected onto the zero section. 
\end{claim}

More generally we have the following: 

For $n_i \in \Z$, let $(a_1^{n_1}\cdots a_\kappa^{n_\kappa},\sigma)$ denote intersections between $\phi^1_{H_V}(T^*_{q_i} T^2)$ and $N^* \alpha_{\sigma(i)}$ for $i = 1, \dots, \kappa$ whose corresponding Hamiltonian chords wrap $n_i$ times around the torus in the $(-1,0)$-direction when projected onto the zero section. 

\begin{lemma}
\label{lem:modelcomp}
    The generators of $CF(\sqcup_{i=1}^\kappa \phi^1_{H_V}(T^*_{q_i} T^2),\sqcup_{i=1}^\kappa N^* \alpha_i)$ for points $q_i$ and curves $\alpha_i$ as in the model case are given by \[\Z[(\textstyle \prod_{i=1}^\kappa a_i^{n_i},\sigma) \ | \ \sigma\in S_\kappa, \ n_i \in \Z ].\]  
\end{lemma}

\begin{proof}
We choose coordinates $(q_{i1},q_{i2}, p_{i1}, p_{i2})$ on $T^*_{q_i}T^2$. 
Ignoring the perturbation term $V(t,q)$ for now, let $p_{i1} > 0$ be the smallest such number such that $\phi^1_{H_V}(q_{i1},q_{i2}, p_{i1}, p_{i2}) \cap N^* \alpha_j  \neq \emptyset$. Following the calculation done in our model example, we see that every such intersection occurs at $\phi^1(q_{i1},q_{i2}, p_{i1}+n, 0)$, for some $n\in \Z$. Then we can describe the set of intersections by the form in the lemma. 

Adding the perturbation term has little effect on the Hamiltonian chords, and it can be chosen small enough and such that $\frac{\partial V}{\partial q}$ is small. It follows that each intersection will occur at $\phi^1(q_{i1},q_{i2}, p_{i1}+n, p_{i2}(n))$ where $p_{i2}(n)$ will be small for all $n$. While $n$ may not be an integer, it will be very close to one. Alternatively, one could order all such $n$'s, giving a bijection with $\Z$. 
\end{proof}

\begin{remark}
    One way to visualize the wrapping is to view the image of the cotangent fibers projected onto $T^2$. If the $\alpha$ curves lie in only one component of $S^1 \times S^1$, the Hamiltonian chords look like geodesics from $q_i$ to $\alpha_j$ in the other $S^1$ component; see Figures \ref{fig:generators} and \ref{fig:geodesics}. This will be made more explicit in the next section. 
\end{remark}

\section{Path space and wrapped HDHF}
\label{sec:pathspaceandwrappedhdhf}

We restrict to the case $\kappa = 1$ and make use of an isomorphism of the wrapped Floer homology with the singular homology of a certain path space. We will review the Morse complex in this context in general before specializing to the surface $\Sigma$ and Lagrangian submanifolds of interest. For more details, the reader is encouraged to look at \cite[Sections 2, 3, and 4]{abbondandolo2010floer} . 

\subsection{Dual Lagrangian formulation and perturbed geodesics}

\label{sec:duallagrangian}We recall some facts about Legendre transforms and perturbed geodesics that will help establish a dual Lagrangian formulation from which to do Morse theory. 

Let $H \in C^\infty([0,1] \times T^*M)$ be a Hamiltonian satisfying the classical Tonelli assumptions. The Fenchel transform defines a smooth, time-dependent Lagrangian on $TM$,
$$L(t,q,v):= \max_{p\in T_q^*M}(\langle p, v \rangle) - H(t,q,p), \ (t,q,v)\in [0,1] \times TM.$$
We call this Lagrangian the \textit{Fenchel dual} of our Hamiltonian $H$. Similarly, given a Lagrangian satisfying equivalent assumptions, we can dualize to get the Hamiltonian 
$$H(t,q,p) = \max_{p\in T_q M}(\langle p, v \rangle) - L(t,q,v), \ (t,q,p)\in [0,1] \times T^*M.$$

Moreover, we have a diffeomorphism known as the Legendre transform
$$\mathcal{L}: [0,1] \times TM \longrightarrow [0,1] \times T^*M,\ (t,q,v) \mapsto (t,q,D_vL(t,q,v)),$$
such that $\mathcal{L}(t,q,v)=(t,q,p)$ if and only if $L(t,q,v) = \langle p, v \rangle - H(t,q,p)$. 

We now specialize to the Lagrangian function $L_V:[0,1] \times TM \to \R$ given by 
$$L_V(t,q,v) = \frac{1}{2}|v|^2 - V(t,q).$$
Its Fenchel dual Hamiltonian is then given by $H(t,q,p) = \frac{1}{2}|p|^2 + V(t,q)$, the same Hamiltonian we considered earlier when defining the Floer complex. 

With these functions in mind, there is an equivalence between Hamiltonian orbits $x: [0,1] \to T^*M$ solving $x'(t) = X_H(t,x(t))$ and curves $\gamma: [0,1] \to M$ which are extremals of the Lagrangian action functional $\mathcal{A}_V(\gamma) = \int_0^1 L_V(t,\gamma, \dot{\gamma}) dt$. More precisely, $x$ is a Hamiltonian orbit if and only if $\gamma := \pi_M \circ x$ is an absolutely continuous extremal of $\mathcal{A}_V$. 

We conclude this subsection with a further equivalence, recalling the definition of a perturbed geodesic. 

\begin{definition}
A {\em $V$-perturbed geodesic} $\gamma$ is a map $[0,1] \to M$ such that $$\nabla_{\dot{\gamma}}\dot{\gamma} = - \nabla V,$$
where $\nabla V$ denotes the gradient of $V$ with respect to the metric $g$. 
\end{definition}

The critical points of $\mathcal{A}_V$ are precisely the $V$-perturbed geodesics. Putting the two equivalences together, we have that the Hamiltonian orbits are in bijection with the $V$-perturbed geodesics. We can impose similar boundary conditions to the perturbed geodesics as we do the Hamiltonian chords in the case of our wrapped Floer theory, giving us a map $\mathcal{L}: \mathcal{P}^Q(H) \to \mathcal{P}^Q(L) $, where $\mathcal{P}^Q(H)$ are the Hamiltonian chords with the boundary condition and $\mathcal{P}^Q(L)$ are the $V$-perturbed geodesics satisfying the same boundary condition. Moreover, condition \hyperlink{H0}{(H0)} placed on the Hamiltonian $H$ will ensure that all critical points will be non-degenerate. 

\begin{definition}
We call $\mathcal{L}$ the Legendre transform. If we let $x \in \mathcal{P}^Q(H)$, then 
$$\mathcal{L}(x)(t):= \pi_M \circ x(t).\\$$  

\end{definition}

Figure \ref{fig:geodesics} below shows the perturbed geodesics corresponding to the generator $(a_{2},(12))$ of our model calculation in Section \ref{modelcalc}. 

\begin{figure}[h]
    \centering
    \includegraphics[width=0.75\textwidth]{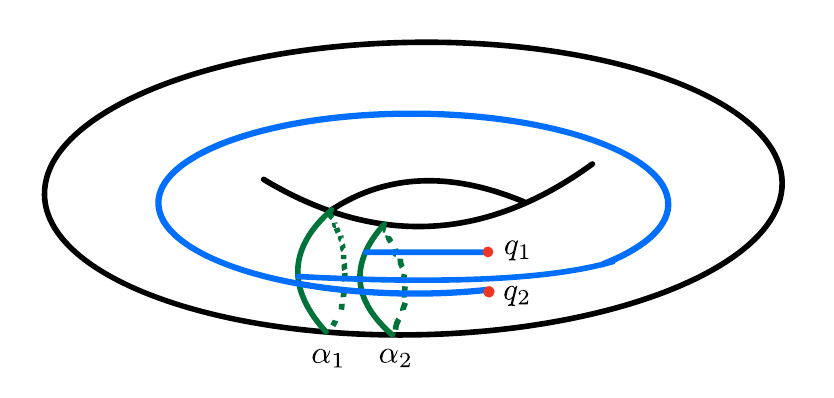}
    \caption{Perturbed geodesics (in blue) representing the generator $(a_{2},(12))$. }
    \label{fig:geodesics}
\end{figure}

\subsection{Path space and wrapped Floer homology}
\hfill\\
\label{sec:AS iso}Let $Q \subset M \times M$ be a closed submanifold as before. Consider the path space $$\Omega_Q(M) = \{ \gamma \in C^0([0,1],M) ~|~(\gamma(0),\gamma(1)) \in Q\}.$$ 
We restrict to a subset consisting of paths in the class $W^{1,2}$, and wish to do Morse theory on this new path space, denoted $\Omega_Q^{1,2}(M)$.

Given the same Hamiltonian $H_V$, Hamiltonian vector field $X_{H_V}$, and flow $\phi^t_{H_V}$ as in the previous section, we define the function $L:[0,1] \times TM \to \R$ satisfying second derivative conditions:

\begin{enumerate}
    \hypertarget{L1}{}
    \item[(L1)] There exists $l_1 > 0$ such that $\nabla_{vv}L(t,q,v)\geq l_1 I$

    \item[(L2)] There exists $l_2 > 0$ such that 
    $$|\nabla_{qq}L(t,q,v)| \leq l_2(1+|v|^2),\ |\nabla_{qv}L(t,q,v)|\leq l_2(1+|v|),\ |\nabla_{vv}L(t,q,v)| \leq l_2,$$
\end{enumerate}
 
 \noindent ensuring that $L$ grows quadratically in the tangent direction. These conditions are equivalent to conditions \hyperlink{H0}{(H1), (H2)} imposed on our Hamiltonian, and will guarantee that $H_V$, the Fenchel dual of $L_V$, satisfies those. We can take $$L_V(t,q,v) = \frac{1}{2}|v|^2 - V(t,q),$$
where $t\in [0,1], q\in M, \text{and }v \in T_q M$.

Considering the Morse function $$\mathcal{A}_V(\gamma) = \int_0^1 L_V(t,\gamma, \dot{\gamma}) dt ,$$ defined for $\gamma \in \Omega_Q^{1,2}(M)$, we define the Morse complex generated by its critical points. This is well defined and the Morse homology $HM_*(\mathcal{A}_V)$ is isomorphic to the singular homology of $\Omega_Q^{1,2}(M)$. Since the inclusion $\Omega_Q^{1,2}(M) \hookrightarrow \Omega_Q(M)$ is a homotopy equivalence, the isomorphism extends to the singular homology of $\Omega_Q(M)$.

\begin{theorem}[Abbondandolo-Schwarz]
\label{theorem 3.3}
There is an isomorphism between the wrapped Floer homology $HF_Q^*(T^*M)$ and the singular homology of the path space $\Omega_Q(M)$.
% Specifically, this holds in the case where $Q= \{pt\} \times S^1 \in T^2 \times T^2$. 
\end{theorem}

\noindent\textbf{For the rest of the paper, we specialize $M$ to a closed oriented surface $\Sigma$ of genus $>0$. } 

Let $Q = \{q\} \times \alpha \in \Sigma \times \Sigma$, where $q\in \Sigma$ is a point and $\alpha$ is a homotopically nontrivial simple closed curve in $\Sigma$. 

\begin{lemma}
\label{lemma 3.4}
     $H_*(\Omega_Q(\Sigma))$ is supported in degree 0. 
\end{lemma}

\begin{proof}
If $\Sigma$ is a torus, we can assume that $g$ is the flat metric, where all $V$-perturbed geodesics with sufficiently small perturbation $V$ are minimal and isolated. If the genus of $\Sigma$ is greater than 1, then we can assume that $g$ is the hyperbolic metric with constant curvature $-1$. It is well known that on a hyperbolic surface, there is a unique $V$-perturbed geodesic in each homotopy class of paths for $V$ sufficiently small. Hence the Morse indices of all critical points are 0. 
\end{proof}

Next let $q_1, \dots, q_\kappa$ be $\kappa$ distinct points on $\Sigma$ and $\alpha_1, \dots, \alpha_\kappa$ be $\kappa$ mutually disjoint homotopically nontrivial simple closed curves on $\Sigma$. We use $T^*_{\bm{q}} \Sigma$ to denote $\sqcup_{i=1}^\kappa T^*_{q_i}\Sigma$ and $N^* \bm{\alpha}$ to denote $\sqcup_{i=1}^\kappa N^* \alpha_i$. 

\begin{corollary}
$HF_Q^*(T^*\Sigma)$ is supported in degree 0. In particular, the grading $|\bm{y}|=0$ for every generator $\bm{y} \in CW(T^*_{\bm{q}}\Sigma,N^*\bm{\alpha})$.
\end{corollary}

\begin{proof}
Each $y_i \in \phi^1_{H_V}(T_{q_i}^* \Sigma) \cap N^* \alpha_j$ corresponds to a time-1 Hamiltonian chord from $T_{q_i}^* \Sigma$ to $N^* \alpha_j$. Its Legendre transform gives a $V$-perturbed geodesic $\gamma$ on $\Sigma$. The Conley-Zehnder index of $y_i$ is equal to the Morse index of $\gamma$ with respect to its Lagrangian action. Lemma \ref{lemma 3.4} above implies that $|y_i|=0$ for all $i$, and so $|\bm{y}|=0$. 
\end{proof}

\subsection{Unordered configuration space and wrapped HDHF}
\hfill\\
\label{sec:initialT1}Let $$\text{UConf}_\kappa(\Sigma) = \{ \{q_1,\dots,q_\kappa \}~|~ q_i \in \Sigma, q_i \neq q_j \ \text{for}\  i \neq j \}$$ be the configuration space of $\kappa$ unordered (distinct) points on $\Sigma$. 

We wish to generalize Theorem \ref{theorem 3.3} to the case where $\kappa > 1$. Let $\bm{q} = \{q_1,\dots,q_\kappa \} \in \text{UConf}_\kappa(\Sigma)$ and $\bm{\alpha} = \alpha_1 \times \cdots \times \alpha_\kappa$, where we view $\bm{\alpha}$ as a subset of $\text{UConf}_\kappa(\Sigma)$ by viewing $(x_1, \dots, x_\kappa)$ as an unordered tuple. The natural analog to consider on the path space side is the path space on $\text{UConf}_\kappa(\Sigma)$, which we denote by $$\Omega(\text{UConf}_\kappa(\Sigma),\bm{q}, \bm{\alpha}) = \{\gamma \in C^0([0,1],\text{UConf}_\kappa(\Sigma))~|~\gamma(0) = \bm{q}, \ \gamma(1) \in \bm{\alpha}\}.$$

We can identify each generator of $CF(\sqcup_{i=1}^\kappa \phi^1_{H_V}(T^*_{q_i} \Sigma),\sqcup_{i=1}^\kappa N^* \alpha_i)$ with a $\kappa$-tuple of $V$-perturbed geodesics from the $q_i$ to the $\alpha_{\sigma(i)}$ using the duality established in Section \ref{sec:AS iso}. This is an element of  $\Omega(\text{UConf}_\kappa(\Sigma),\bm{q}, \bm{\alpha})$.

To make this more precise, we construct an evaluation map 
$$\mathcal{E}: CF(\sqcup_{i=1}^\kappa \phi^1_{H_V}(T^*_{q_i} \Sigma),\sqcup_{i=1}^\kappa N^* \alpha_i) \longrightarrow C_0(\Omega(\text{UConf}_\kappa(\Sigma),\bm{q}, \bm{\alpha})) \otimes \Z[[\hbar]],$$
\noindent where $C_0(\Omega(\text{UConf}_\kappa(\Sigma),\bm{q}, \bm{\alpha}))$ is the space of 0-chains of the path space $\Omega(\text{UConf}_\kappa(\Sigma),\bm{q}, \bm{\alpha})$. 

The map $\mathcal{E}$ counts pseudo-holomorphic curves between the conormal Lagrangians and the zero section. Parametrizing the boundary of the curves along the zero section produces a path in the unordered configuration space. We keep the parameter $\hbar$ around to track the Euler characteristic of the map, which will later relate to the HOMFLY skein relation on braids.

Let $T_1 := D_2$ be our $A_\infty$-base where $\partial_iT_1 = \partial_i D_2$; shown in Figure \ref{fig:T1}. Let $\mathcal{T}_1$ be the moduli space of $T_1$ modulo automorphisms, and choose representatives $T_1$ of equivalence classes in a smooth manner. Let $\pi_{T^*\Sigma}$ be the projection $T_1 \times T^*\Sigma \to T^*\Sigma$ and choose a sufficiently generic consistent collection of compatible almost complex structures such that they are close to a split almost complex structure projecting holomorphically to $T_1$, as in Section \ref{sec:HDHFreview}. 
We denote by $\mathcal{H}(\bm{q}',\bm{y},\ \bm{x})$ the moduli space of maps $$u: (\dot{F}, j) \longrightarrow (T_1 \times T^*\Sigma, J_{T_1}),$$ where $(F, j)$ is a compact Riemann surface with boundary,  $\bm{p}_0$, $\bm{p}_1$, $\bm{p}_2$ are disjoint tuples of boundary punctures of $F$ and $\dot{F} = F \setminus \cup_i$\textbf{$p_i$}, satisfying: 

\begin{equation}
    \begin{cases}
    du \circ j = J_{T_1} \circ du;\\
    \pi_{T^*\Sigma} \circ u(z) \in \phi^1_{H_V}(\sqcup_{i=1}^\kappa T^*_{q_i}\Sigma) \text{ if } \pi_{T_1} \circ u(z) \subset \partial_0 T_1;\\
    
    \text{each component of } \partial \dot{F}\text{ that projects to }\partial_0T_1 \text{ maps to a distinct }\phi^1_{H_V}(T^*_{q_i}\Sigma) ;\\
    
    \pi_{T^*\Sigma} \circ u(z) \in \sqcup_{i=1}^\kappa N^* \alpha_i \text{ if } \pi_{T_1} \circ u(z) \subset \partial_1 T_1;\\
     
    \text{each component of } \partial \dot{F}\text{ that projects to }\partial_1T_1 \text{ maps to a distinct }
     N^*\alpha_i;\\

    \pi_{T^*\Sigma} \circ u(z) \in \Sigma \text{ if } \pi_{T_1} \circ u(z) \subset \partial_2 T_1;\\
     
    \pi_{T^*\Sigma} \circ u \text{ tends to } \bm{q}',\bm{y}, \ \bm{x} \text{ as } s_0,s_1 ,s_2 \to +\infty ;\\
   
    \pi_{T_1} \circ u \text{ is a }\kappa\text{-fold branched cover of a fixed }T_1 \in \mathcal{T}_1.\\
    
    \end{cases}
\end{equation}

\begin{figure}[h]
    \centering
    \includegraphics[width=0.4\textwidth]{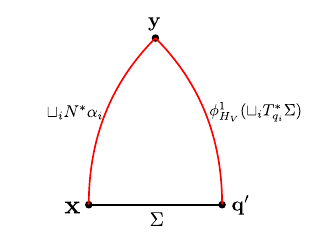}
    \caption{The $A_\infty$-base $T_1$}
    \label{fig:T1}
\end{figure}

In simpler terms, we look at the moduli space of holomorphic curves between the Lagrangians involved and the zero section of $T^* \Sigma$ in the framework of HDHF with only positive punctures.

\begin{remark}
While the intersections $\bm{y}$ and $\bm{q}'$ are discrete, we have an $S^1$-worth of choices for each $x_i \in \alpha_i \cap \Sigma$. To resolve this issue, we can either use the Morse-Bott (clean intersection) formalism or perturb the zero section near the $N^* \alpha_i$. This results in two intersection points $x_1,x_2 \in N^*\alpha_i \cap (\Sigma \times \{0\})$. If we are interested in index 0 curves in our moduli space then only one of these intersections will have the right grading. This is shown in Figure \ref{fig:bottomgenerator} for our model $\Sigma=T^2,\kappa=1$ case by counting pseudo-holomorphic triangles bounded by the conormal Lagrangians and the perturbed zero section. 

\end{remark}

\begin{figure}[h]
    \centering
    \includegraphics[width=0.5\textwidth]{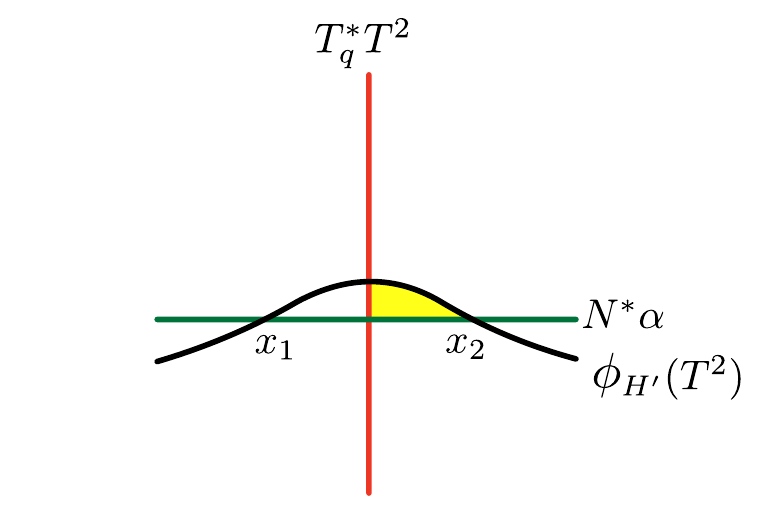}
    \caption{Perturbed zero section in one of the $T^*S^1$ directions. Only $x_2 \in N^* \alpha \cap \phi_{H'}(T^2)$ has the right Maslov index.}
    \label{fig:bottomgenerator}
\end{figure}

Let $\mathcal{H}^\chi(\bm{q}',\bm{y},\bm{x})$ be the subset of $\mathcal{H}(\bm{q}',\bm{y},\bm{x})$ such that $\chi(\dot F) = \chi$. Moreover, let $$\mathcal{H}^{\chi}(\bm{q}',\bm{y},\bm{\alpha})= \sqcup_{\bm{x}\in \bm{\alpha}} \mathcal{H}^{\chi}(\bm{q}',\bm{y},\bm{x}).$$
\begin{lemma}
    For fixed generic $J_{T_1}$, $\mathcal{H}^{\chi}(\bm{q}',\bm{y},\bm{x})$ is of dimension 0 and consists of discrete regular curves for all $\bm{q}',\  \bm{y},$ and $\bm{x}$ such that $\bm{x}$ is a tuple of bottom generators for the Morse-Bott intersections $N^*\alpha_i \cap (\Sigma \times \{0\})$.
\end{lemma}

\begin{proof}
    By Corollary 3.5, we have $|\bm{y}|=0$. Computing the grading for intersections of cotangent fibers with the zero section gives that $|\bm{q}'|=0$; see \cite{honda2022higher}. Since the $\bm{x}$ are bottom generators for our Morse-Bott intersection, it follows that $|\bm{x}|=0$ and so the virtual dimension of $\mathcal{H}^{\chi}(\bm{q}',\bm{y},\bm{x})$ is 0. The rest follows from standard transversality arguments. 
\end{proof}

\begin{lemma}
    Given $\bm{q}',\bm{y},\bm{\alpha}$, the moduli space $\mathcal{H}^{\chi}(\bm{q}',\bm{y},\bm{\alpha})$ consists of finitely many curves for each Euler characteristic $\chi$.
\end{lemma}

\begin{proof}
    Each $\bm{y}$ determines a unique $\bm{q}'$ and $\bm{x} \in \bm{\alpha}$. The energy bound along with Gromov compactness gives the result. 
\end{proof}

Fix a parametrization of the arc $\partial_2 T_1$ from $p_0$ to $p_2$ by $\tau: [0,1] \to \partial_2 T_1$. There exists a sufficiently generic consistent collection of almost complex structures such that for all $u \in \mathcal{H}(\bm{q}',\bm{y},\bm{\alpha})$,  $(\pi_{\Sigma}\circ u)\circ(\pi_{T_1 \circ u})^{-1}\circ\tau(t)$ consists of $\kappa$ distinct points on $\Sigma$ for each $t\in [0,1]$ and hence gives a path in $\text{UConf}_\kappa(\Sigma)$:
$$\gamma(u):[0,1] \longrightarrow \text{UConf}_\kappa(\Sigma),$$ $$t\mapsto (\pi_{\Sigma}\circ u)\circ(\pi_{T_1 \circ u})^{-1}\circ\tau(t). $$
Since $\gamma(0) = \bm{q}'$ and $\gamma(1) \in \bm{\alpha}$, it follows that $\gamma(u) \in \Omega(\text{UConf}_\kappa(\Sigma),\bm{q}',\bm{\alpha})$.

Define the evaluation map 
$$\mathcal{E}: CW(\sqcup_{i=1}^\kappa T_{q_i}^*\Sigma,\sqcup_{i=1}^\kappa N^* \alpha_i) \longrightarrow C_0(\Omega(\text{UConf}_\kappa(\Sigma) , \bm{q}', \bm{\alpha}))  \otimes \Z[[\hbar]] $$
$$\bm{y} \mapsto \sum_{u \in \mathcal{H}(\bm{q}',\bm{y},\bm{\alpha})} (-1)^{\natural(u)} \cdot  \hbar^{\kappa - \chi(u)} \cdot \gamma(u),$$
where $(-1)^{\natural(u)}$ is the sign assigned to $u$.

Since the perturbation term $V$ has small $W^{1,2}$-norm, the Hamiltonian vector field $X_{H_V}$ has small norm near the zero section $\Sigma$ and hence $\bm{q}'$ is close to $\bm{q}$. We choose non-intersecting short paths $\gamma_i$ on $\Sigma$ from $q_i$ to $q'_i$ for $i = 1, \dots, \kappa$. Pre-concatenating with $\{\gamma_i \}$ allows us to identify $\Omega(\text{UConf}_\kappa(\Sigma) , \bm{q}', \bm{\alpha})$ with $\Omega(\text{UConf}_\kappa(\Sigma) , \bm{q}, \bm{\alpha})$. We make this identification whenever possible. 

Next, we have a projection
$$\mathcal{P}:C_0(\Omega(\text{UConf}_\kappa(\Sigma) , \bm{q}, \bm{\alpha}))  \otimes \Z[[\hbar]] \to (H_0(\Omega(\text{UConf}_\kappa(\Sigma) , \bm{q}, \bm{\alpha}))  \otimes \Z[[\hbar]]) / \text{HOMFLY skein}$$
given by first taking the homotopy class of the path $\gamma(u)$ and then, viewing $\gamma(u)$ as a braid, quotienting by the HOMFLY skein relation (given in Definition \ref{def:braidskeinalgebra}). 

Composing the evaluation map and the projection, we arrive at the map $\mathcal{F} = \mathcal{P} \circ \mathcal{E}$. 

Let $BSk_\kappa(\Sigma,\bm{q},\bm{\alpha})$ denote the free $\Z[[\hbar]]$-module generated by homotopy classes of braids $\gamma \in \Omega(\text{UConf}_\kappa(\Sigma) , \bm{q}, \bm{\alpha})$ modulo the HOMFLY skein relation. Then
$$\mathcal{F}:CW(\sqcup_{i=1}^\kappa T_{q_i}^*\Sigma,\sqcup_{i=1}^\kappa N^* \alpha_i) \longrightarrow BSk_\kappa(\Sigma,\bm{q}, \bm{\alpha})$$ is given by 
$$\bm{y} \mapsto \sum_{u \in \mathcal{H}(\bm{q}',\bm{y},\bm{\alpha})} (-1)^{\natural(u)} \cdot  \hbar^{\kappa - \chi(u)} \cdot [\gamma(u)],$$
where $[\gamma(u)]$ is viewed as an equivalence class of braids modulo the HOMFLY skein relation.

\section{HDHF with conormal boundary conditions as a braid skein algebra module}

In this section we recall the equivalence between wrapped HDHF for cotangent fibers and the braid skein algebra of a surface. The reader is encouraged to look at \cite{honda2022higher} and \cite{morton2021dahas} for a more detailed exposition of this section's content. Let $\Sigma$ be a closed oriented surface of genus $>0$ and $ q_1,\dots, q_\kappa \in \Sigma$ be distinct points. 

\subsection{The braid skein algebra of a surface}
\hfill\\
Consider the braid group $B_\kappa(\Sigma \setminus \{ \ast\}, \bm{q})$ of $\kappa$-braids in the punctured surface $\Sigma \setminus \{\ast \}$ based at $\bm{q} = \{q_1,\dots, q_\kappa\}$. One way to view this is to take the thickened surface $\Sigma \times I$ with a fixed base string $\{\ast\} \times I$. In this case, the elements are made up of $\kappa$ strings oriented monotonically from $\Sigma \times \{0\}$ to $\Sigma \times \{1\}$ which do not intersect each other or the base string. Two braids are equivalent if they are isotopic to each other, with the isotopy avoiding the base string. 

% We now choose a convenient basis for these braids. Viewing $T^2$ as a square $I \times I$ with opposite sides identified, we choose $\ast = (\frac{1}{2},\frac{1}{2})$ and let our $\kappa$ points $q_1,\dots,q_\kappa$ line up in increasing fashion along the lower part of the diagonal from $(0,0)$ to $\ast$. Let $x_i \ (\text{respectively, } y_i)$ be the braid which consists of the point $q_i$ moving uniformly around the $(-1,0)\ (\text{respectively, } (0,1))$ curve. Let $\sigma_i$ for $1\leq i\leq \kappa-1$ be the braid which locally exchanges the strings from $q_i$ and $q_{i+1}$ in a counterclockwise direction when looking down onto $T^2$, as shown in Figure \ref{fig:sigmabraid} below. 

% \begin{figure}[h]
%     \begin{subfigure}{.5\textwidth}
%     \centering
%     \includegraphics[width=1\textwidth]{Figures/xybraids.png}
%     \caption{Braids $x_i$ and $y_i$ }
%     \label{fig:xybraids}
%     \end{subfigure}%
% \begin{subfigure}{.5\textwidth}
%     \centering
%     \includegraphics[width=0.65\textwidth]{Figures/sigmabraid.png}
%     \caption{Braid $\sigma_i$}
%     \label{fig:sigmabraid}
% \end{subfigure}
% \caption{Generators for the braid group on the punctured torus}

% \end{figure}

% Instead of giving the presentation for the braid group in the punctured torus, we build on this with an eye towards the braid skein algebra. 

\begin{definition}
\label{def:braidskeinalgebra}
    The {\em braid skein algebra $BSk_\kappa(\Sigma,\bm{q},\ast)$ (or the surface Hecke algebra) of the surface $\Sigma$} is the free $\Z[s^{\pm 1},c^{\pm 1}]$-module generated by $\kappa$-braids in the punctured surface $\Sigma \setminus \{\ast\}$ based at $\bm{q}$, up to isotopy which does not intersect $\{\ast\} \times [0,1]$, subject to the local relations:

   \begin{enumerate}
       \item the HOMFLY skein relation
       \begin{figure}[h]
    \centering
    \includegraphics[width=0.5\textwidth]{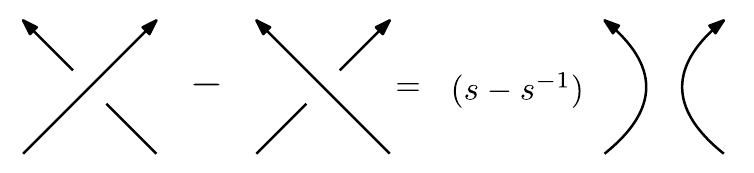}
    
    \label{fig:skein}
\end{figure}

\item the marked point relation
\\

\begin{figure}[h!]
    \centering
    \includegraphics[width=0.25\textwidth]{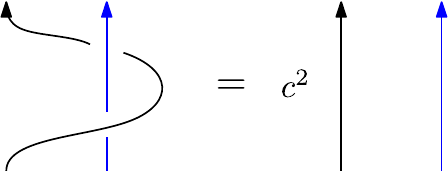}
    
    \label{fig:markedpoint}
\end{figure}

    \noindent where the string in blue corresponds to the marked point strand $\{\ast\} \times [0,1]$.

   \end{enumerate}
       
\end{definition}

Using the marked point relation, we define a slightly more useful \textit{$c$-deformed homotopy relation}. 
\begin{definition}
    The { \em $c$-deformed braid group  $B_\kappa(\Sigma,\bm{q})_c$ of $\Sigma$ based at $\bm{q}$} is generated by $B_\kappa(\Sigma \setminus \{\ast \},\bm{q})$ and a central element $c$, subject to the following $c$-deformed homotopy relation:
\begin{equation}
    [\gamma_2] = c^{2\langle H,\ast \rangle}[\gamma_1],
\end{equation}
where $\gamma_1,\gamma_2 \in \Omega(\text{UConf}_\kappa(\Sigma \setminus \{\ast \}),\bm{q})$, $H$ is a homotopy between $\gamma_1$ and $\gamma_2$, $\langle H, \ast \rangle$ is the algebraic intersection number for the homotopy $H$ defined in \cite{honda2022higher}, and $[\gamma_i]$ is the homology class of the braid $\gamma_i$. 
    
\end{definition}

The quotient of $\Z[s^{\pm 1}][B_\kappa(\Sigma,\bm{q})_c]$ by the HOMFLY skein relation gives the braid skein algebra $BSk_\kappa(\Sigma, \bm{q}, \ast)$.

\subsection{Hecke algebra realization in HDHF}
\hfill\\
\label{sec:HeckeHDHF}We briefly review the connection between the wrapped HDHF of contangent fibers on surfaces and the surface Hecke algebras defined in \cite{honda2022higher}. 

\begin{definition}
    The {\em wrapped HDHF chain complex of disjoint cotangent fibers with parameter $c$}, $CW(\sqcup_{i=1}^\kappa T^*_{q_i}\Sigma)_c$, is given by $CF(\phi^1_{H_V}(\sqcup_{i=1}^\kappa T^*_{q_i}\Sigma),\sqcup_{i=1}^\kappa T^*_{q_i}\Sigma)\otimes \Z[c^{\pm 1}]$ as a $\Z$-module and has enhanced $A_\infty$-operations to include $c$-coefficients. Specifically, 
    $$\mu^m(\bm{y}_1,\dots,\bm{y}_m) = \sum_{u \in \mathcal{M}^{\text{ind}=0}(\bm{y}_1,\dots,\bm{y}_m,\bm{y}_0)} (-1)^{\natural(u)} \cdot c^{2\langle u, \ast \rangle} \cdot \hbar^{\kappa - \chi(u)}\cdot \bm{y}_0.$$
\end{definition}
    
    Given a map $u \in \mathcal{M}(\bm{y}_1,\dots,\bm{y}_m,\bm{y}_0)$, $\langle u,\ast \rangle$ is the intersection number of $\pi_\Sigma(u)$ and $\ast$, after some modifications to $\pi_\Sigma(u)$ to ensure this is well defined. 
    
\begin{proposition}[Honda-Tian-Yuan]
    The $A_\infty$-algebra $CW(\sqcup_{i=1}^\kappa T^*_{q_i}\Sigma)_c$ is supported in degree zero, and hence is an ordinary algebra. 
\end{proposition}

With this on hand, HTY go on to prove the following theorem connecting the wrapped HDHF to the surface Hecke algebra $H_\kappa(\Sigma,\bm{q})$. 

\begin{theorem}[Honda-Tian-Yuan]
    There is an isomorphism of algebras 
    $$\mathcal{F}:CW(\sqcup_{i=1}^\kappa T^*_{q_i}\Sigma)_c \longrightarrow H_\kappa(\Sigma,\bm{q}) \otimes \Z[[\hbar]],$$
    where the surface Hecke algebra $H_\kappa(\Sigma,\bm{q})$ is naturally isomorphic to the braid skein algebra $BSk_\kappa(\Sigma,\bm{q},\ast)$.
\end{theorem}

We run through an (informal) overview of the proof as the main ideas will be modified and applied in the proof of Theorem \ref{thm:Fiso1} in Section \ref{sec:Fiso proof}. 

The map $\mathcal{F}$ is an evaluation map similar to that in Section \ref{sec:initialT1}. It is constructed by considering a moduli space of pseudo-holomorphic curves bounded by $\phi^1_{H_V}(T^*_{\bm{q}}\Sigma),\  T^*_{\bm{q}}\Sigma,$ and $\Sigma$. The authors then make use of the construction of the braid skein algebra coming from the unordered configuration space to relate this to the surface Hecke algebra. 

The next step is to show that $\mathcal{F}$ is indeed a homomorphism of algebras. This is done by taking the moduli space of index 1 curves projecting to a 4-punctured disk with boundary conditions given by $\phi^2_{H_V}(T^*_{\bm{q}}\Sigma),$ $\phi^1_{H_V}(T^*_{\bm{q}}\Sigma),$ $T^*_{\bm{q}}\Sigma,$ and $\Sigma$. An inspection of the boundary of the compactification of this space concludes that $\mathcal{F}$ respects the algebra structure. More specifically, the only breakings that occur are ones that correspond to $\mathcal{F}(\mu^2(\bm{y,y'}))$ or $\mathcal{F}(\bm{y})\mathcal{F}(\bm{y}')$. 

The final step is to show that $\mathcal{F}|_{\hbar=0}$ is an isomorphism and then use the algebra homomorphism properties to prove it is a bijection when $\hbar$ is reintroduced. 

We will repeat and modify much of the argument in our proof of Theorem \ref{thm:Fiso1}.

\subsection{The parameter $c$}\label{sec:parameterc}
\hfill\\
We enhance $CW(\sqcup_{i=1}^\kappa T^*_{q_i}\Sigma,\sqcup_{i=1}^\kappa  N^*\alpha_i)$ to include the parameter $c$. Let $$CW(\sqcup_{i=1}^\kappa T_{q_i}^*\Sigma,\sqcup_{i=1}^\kappa N^* \alpha_i)_c := CW(\sqcup_{i=1}^\kappa T_{q_i}^*\Sigma,\sqcup_{i=1}^\kappa N^* \alpha_i) \otimes \Z[c^{\pm 1}].$$ Consider the updated evaluation map: 
$$\mathcal{E}: CW(\sqcup_{i=1}^\kappa T_{q_i}^*\Sigma,\sqcup_{i=1}^\kappa N^* \alpha_i)_c \longrightarrow C_0(\Omega(\text{UConf}_\kappa((\Sigma) , \bm{q}, \bm{\alpha})) \otimes \Z[c^{\pm 1}]  \otimes \Z[[\hbar]] $$
given by 
$$\bm{y} \mapsto \sum_{u \in \mathcal{H}(\bm{q}',\bm{y},\bm{\alpha})} (-1)^{\natural(u)} \cdot c^{2\langle u, \ast \rangle} \cdot  \hbar^{\kappa - \chi(u)} \cdot \gamma(u),$$

\noindent where $\langle u, \ast \rangle := \langle [\pi_{\Sigma}(u)]',\ast \rangle$ is defined in \cite[Section 5]{honda2022higher} .  

\begin{definition}
\label{def:c-deformedbraidskeinalg}
The {\em braid skein group $BSk_\kappa(\Sigma,\bm{q},\bm{\alpha}, \ast)$ from $\bm{q}$ to $\bm{\alpha}$ on $\Sigma \setminus \{\ast\}$}
is the free $\Z[[\hbar]]$-module generated by $c$-deformed homotopy classes of paths in $$\Omega(\text{UConf}_\kappa(\Sigma),\bm{q}, \bm{\alpha}) = \{\gamma \in C^0([0,1],\text{UConf}_\kappa(\Sigma))~|~\gamma(0) = \bm{q}, \ \gamma(1) \in \bm{\alpha} = \alpha_1 \times \dots \times \alpha_\kappa\}$$ modulo the HOMFLY skein relation.     
\end{definition}

\begin{remark}
    Given a braid $\gamma$, its $c$-deformed homotopy class $[\gamma]_c$ is the set of braids equivalent to it under the local relations given in Definition \ref{def:braidskeinalgebra}. Using these relations, every braid is equivalent to a $\Z[\hbar,c^{\pm 1}]$-combination of $\kappa$-tuples of perturbed geodesics. In the spirit of \cite{abbondandolo2010floer}, we can flow the braid $\gamma$ by the Morse function defined in Section \ref{sec:geoaction}. Every time a crossing switches from positive to negative, or vice versa, the flow will bifurcate according to the HOMFLY skein relation. Similarly, crossing the strand over the puncture $\ast$ will pick up factors of $c$. Continuing this flow, we arrive at a $\Z[\hbar,c^{\pm 1}]$-combination of $\kappa$-tuples of perturbed geodesics. Since the relations of our $c$-deformed homotopy classes are the same as those during the Morse flow, the two equivalence classes will be the same. 
\end{remark}

 Composing $\mathcal{E}$ with the projection which takes the $c$-deformed homotopy class and quotients out by the HOMFLY skein relation gives the map 
\begin{equation}
\label{eq:Ffull}
    \mathcal{F}:CW(\sqcup_{i=1}^\kappa T_{q_i}^*\Sigma,\sqcup_{i=1}^\kappa N^* \alpha_i)_c \longrightarrow BSk_\kappa(\Sigma,\bm{q},\bm{\alpha}, \ast).
\end{equation}

\subsection{The proof of Theorem \ref{thm:Fiso1}, $\kappa=1$ case}\label{sec:proofofthm1.1}
\hfill\\
In this subsection we prove Theorem \ref{thm:Fiso1} for $\kappa=1$. Let $\mathcal{F}_0$ denote the specialization $\mathcal{F}_{\hbar=0}$.

We start by briefly reviewing the chain map
$$\Theta: CM_\ast(\Omega^{1,2}(\Sigma,q,\alpha),\mathcal{A}_V) \longrightarrow CW(T_q^* \Sigma, N^* \alpha)$$ constructed by Abbondandolo and Schwarz in \cite[Theorem 3.3]{abbondandolo2010floer}, where the domain $CM_\ast(\Omega^{1,2}(\Sigma,q,\alpha),\mathcal{A}_V)$ is the Morse complex of the function $\mathcal{A}_V$ defined by: $$\mathcal{A}_V(\gamma) = \int_0^1 L_V(t,\gamma, \dot{\gamma}) dt ,$$ for $\gamma \in \Omega^{1,2}(\Sigma,q,\alpha)$. Both complexes are concentrated at grading 0, so $\Theta$ is an isomorphism from the group generated by index 0 critical points of $\mathcal{A}_V$ to the wrapped Floer group $CW(T_q^* \Sigma, N^* \alpha)$.

In what follows, we identify $CM_0(\Omega^{1,2}(\Sigma,q,\alpha),\mathcal{A}_V)$ with $CM_0(\Omega(\Sigma,q,\alpha),\mathcal{A}_V)$.

Given $y \in CM_0(\Omega(\Sigma,q,\alpha),\mathcal{A}_V), x \in CW(T_q^* \Sigma, N^* \alpha)$, AS construct the space $\mathcal{M}(y,x)$ of maps 
$$u: (-\infty, 0] \times [0,1] \longrightarrow T^*\Sigma$$
solving the Floer equation \eqref{Floereq}, which converge to the Hamiltonian chord representing $x$ at $-\infty$, satisfy boundary conditions $T^*_q\Sigma$ along $(-\infty,0] \times \{0\}$, $N^* \alpha$ along $(-\infty,0] \times \{1\}$, and such that the image of $u(\{0\} \times [0,1])$ under the projection from $T^* \Sigma$ to $\Sigma$ is a path on $\Sigma$ lying in the descending manifold of $y$ with respect to the negative gradient flow of $\mathcal{A}_V$. 

Letting $\#\mathcal{M}(y,x)$ be the count of such maps, we have $$\Theta(y) = \sum \#\mathcal{M}(y,x) x.$$

Let $$\Theta_c: CM_\ast(\Omega(\Sigma,q,\alpha),\mathcal{A}_V) \otimes \Z[c^{\pm 1}] \longrightarrow CW(T_q^* \Sigma, N^* \alpha)_c$$ be the $\Z[c^{\pm 1}]$-linear extension of $\Theta$.\\

We make an additional identification which allows us to apply results of \cite{abbondandolo2010floer}. 

\begin{proposition}
\label{prop:Morsebraidskeiniso}
    $BSk_{\kappa=1}(\Sigma,q,\alpha, \ast) \simeq CM_0(\Omega(\Sigma,q,\alpha),\mathcal{A}_V) \otimes \Z[c^{\pm 1}] $
\end{proposition}

\begin{proof}
Elements in $BSk_{\kappa=1}(\Sigma,\bm{q},\bm{\alpha}, \ast)$ are $c$-deformed homotopy classes of paths starting at $q$ and ending on $\alpha$. Every such path is homotopic to a unique perturbed geodesic on $\Sigma$ with the same boundary conditions. Let $\gamma_1$ be a path homotopic to a perturbed geodesic $\gamma_2$ and let $n$ be the signed intersection number of this homotopy with the marked point $\ast$. The map sending $[\gamma_1]_c \mapsto c^{2n}\gamma_2$ is the desired isomorphism. 
\end{proof}

We prove the following proposition, recreating a proof similar to Abouzaid's \cite{abouzaid2012wrapped} in the case of cotangent fibers. 
Let $\tilde{\mathcal{F}}_0$ be the composition of $\mathcal{F}_0$ with the isomorphism described in Proposition \ref{prop:Morsebraidskeiniso}.

\begin{proposition}
\label{Proposition 4.10}
Let $\kappa=1$. The map $$\tilde{\mathcal{F}}_0: CW(T_q^* \Sigma, N^* \alpha)_c \longrightarrow CM_0(\Omega(\Sigma,q,\alpha),\mathcal{A}_V) \otimes \Z[c^{\pm 1}] $$ is an isomorphism of chain complexes. Moreover, $\tilde{\mathcal{F}}_0$ is an inverse to $\Theta_c$. 
\end{proposition}

\begin{proof}
Since $\Theta$ is an isomorphism of chain complexes, it suffices to show that the composition 
$$CM_0(\Omega(\Sigma,q,\alpha),\mathcal{A}_V) \otimes \Z[c^{\pm 1}]\overset{\Theta_c}\longrightarrow CW(T_q^* \Sigma, N^* \alpha)_c \overset{\tilde{\mathcal{F}}_{0}}\longrightarrow CM_0(\Omega(\Sigma,q,\alpha),\mathcal{A}_V) \otimes \Z[c^{\pm 1}]$$
is homotopic to the identity map on $CM_0(\Omega(\Sigma,q,\alpha),\mathcal{A}_V) \otimes \Z[c^{\pm 1}]$. We show that it is not just homotopic to the identity, but is in fact the identity map on the nose. It follows then that $\tilde{\mathcal{F}}_{0} = \Theta^{-1}_c$.
Let $y$ and $z$ be critical points of the Lagrangian actional $\mathcal{A}_V$ which is Fenchel dual to $H_V$ and let $a$ be a positive real number.  

Define $\mathcal{C}(y,z;a)$ to be the moduli space of maps  
$$u: [0,a] \times [0,1] \longrightarrow T^*\Sigma$$ which solve the Floer equation \eqref{Floereq} with boundary conditions $T^*_q\Sigma$ along $[0,a] \times \{0\}$, $N^* \alpha$ along $[0,a] \times \{1\}$, and such that 
$u(\{0\} \times [0,1])$ is contained in the zero section and, considered as a path on $\Sigma$, lies on the ascending manifold of $z$, while the image of $u(\{a\} \times [0,1])$ under the projection from $T^* \Sigma$ to $\Sigma$ is a path on $\Sigma$ lying on the descending manifold of $y$ with respect to the the negative gradient flow of $\mathcal{A}_V$. This is shown in the central part of Figure \ref{fig:inverse}. 

Write $\mathcal{C}(y,z) := \sqcup_{a \in [0,\infty)} \mathcal{C}(y,z;a)$ and let $\overline{\mathcal{C}}(y,z)$ be its Gromov compactification. 

When $a=0$, any solution in $\mathcal{C}(y,z;a)$ is necessarily constant. Thus, the count of rigid elements of $\mathcal{C}(y,z;0)$ gives the identity map on $CM_0(\Omega(\Sigma,q,\alpha),\mathcal{A}_V) \otimes \Z[c^{\pm 1}]$. 

Letting $a$ go to $+\infty$, a family of maps $u_a$ defined on finite strips $[0,a] \times [0,1]$ breaks into two maps $u_-, u_+$ defined on semi-infinite strips, as shown in the bottom right part of the figure below. This boundary component is precisely the composition $\tilde{\mathcal{F}}_{0} \circ \Theta$.

The remaining boundary strata occur at finite $a$ when the projection of $u(\{a\} \times [0,1])$ to $\Sigma$ escapes to the ascending manifold of a critical point $y'$ which differs from $y$. Similarly, the image of $u(\{0\} \times [0,1])$ may converge to the descending manifold of a critical point $z' \neq z$. Such boundary strata are in bijective correspondence with 
$$\mathcal{T}(y,y') \times \overline{\mathcal{C}}(y',z) \cup \overline{\mathcal{C}}(y,z')\times \mathcal{T}(z',z),$$
where $\mathcal{T}(y,y')$ is the moduli space of gradient trajectories from $y$ to $y'$.

For a general manifold this would give a chain homotopy. However, all critical points of $\mathcal{A}_V$ have grading 0 and so there are no gradient trajectories between critical points. Thus there are no boundary strata of this form.

It follows then that, up to signs, $\mbox{id} = \tilde{\mathcal{F}}_{0} \circ \Theta$, where $\mbox{id}$ is the identity map on $$CM_0(\Omega(\Sigma,q,\alpha),\mathcal{A}_V) \otimes \Z[c^{\pm 1}].$$
Moreover, identifying $x\in CW(T_q^* \Sigma, N^* \alpha)_c$ with a time-1 Hamiltonian chord, $\tilde{\mathcal{F}}_{0}$ maps $x$ to the homotopy class of its Legendre transform $[\mathcal{L}(x)]$. 
\end{proof}

\begin{figure}[h]

    \centering
    \includegraphics[width=0.4\textwidth]{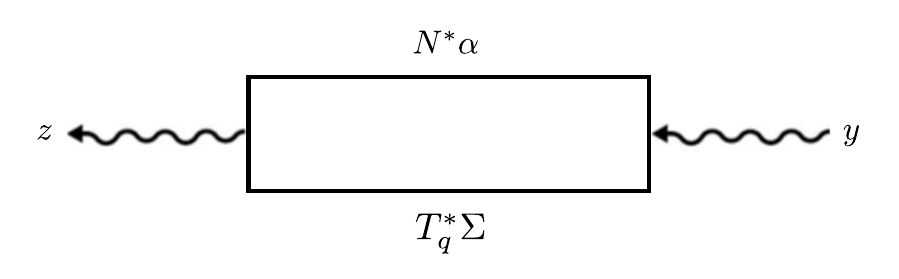}

\centering
\begin{subfigure}{.5\textwidth}
  \centering
  \includegraphics[width=.9\linewidth]{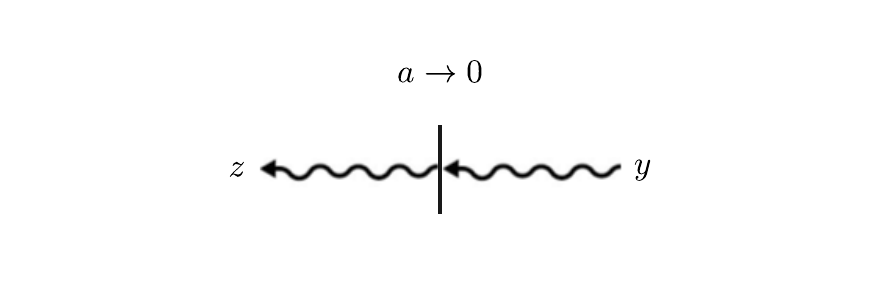}
  \caption{$a=0$ boundary}
  \label{fig:inversea0}
\end{subfigure}%
\begin{subfigure}{.5\textwidth}
  \centering
  \includegraphics[width=.8\linewidth]{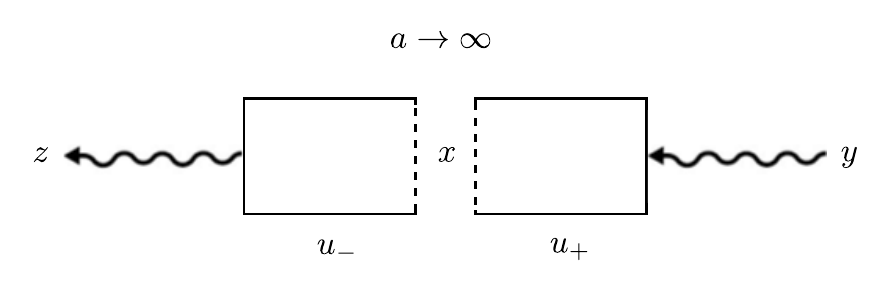}
  \caption{$a\to \infty$ boundary}
  \label{fig:inverseinfinity}
\end{subfigure}
\caption{$\overline{\mathcal{C}}(y,z;a)$ with its boundary. The arrows represent downward gradient flows.} 
\label{fig:inverse}
\end{figure}

 Composing $\tilde{\mathcal{F}}_0$ with the isomorphism  $$CM_0(\Omega(\Sigma,q,\alpha),\mathcal{A}_V) \otimes \Z[c^{\pm 1}] \hookrightarrow BSk_1(\Sigma,q,\alpha, \ast)$$
 given by viewing a geodesic as a path with the same boundary conditions proves that $\mathcal{F}_{0,\kappa=1}$ is an isomorphism.

\subsection{The proof of Theorem \ref{thm:Fiso1}, general case}
\hfill\\
\label{sec:Fiso proof}We first extend Proposition \ref{Proposition 4.10} to $\kappa\geq1$.

\begin{lemma}
\label{lem:hbar0}
    Let $\kappa \geq 1$. Then $\mathcal{F}_0$ is an isomorphism. 
\end{lemma}

\begin{proof}
    Let $\kappa \geq 1$ and $\hbar=0$. $BSk_{\kappa}(\Sigma,\bm{q},\bm{\alpha}, \ast)_c|_{\hbar=0}$ no longer keeps track of braid crossings and thus is generated over $\Z[c^{\pm 1}]$ by elements of the form $\otimes_{i=1}^\kappa H_0(\Omega(\Sigma,q_i,\alpha_{\rho(i)}))_c = \otimes_{i=1}^\kappa BSk_1(\Sigma,q_i,\alpha_{\rho(i)},\ast)_c$, where $\rho \in S_{\kappa}$. Each homotopy class of paths in $H_0(\Omega(\Sigma,q_i,\alpha_{\rho(i)}))_c$ contains a unique $V$-perturbed geodesic from $q_i$ to $\alpha_{\rho(i)}$. Applying the Fenchel duality in Section \ref{sec:duallagrangian} gives us a bijection between $V$-perturbed geodesics and their dual Hamiltonian chords. By Proposition \ref{Proposition 4.10}, this bijection is given by $\mathcal{F}_{0,\kappa=1}$ after identifying generators of $CW(\Sigma,q_i,\alpha_{\rho(i)})$ with Hamiltonian chords from $q_i$ to $\alpha_{\rho(i)}$. 

    Since $\hbar = 0$, the only maps $u$ contributing to $\mathcal{F}_0$ are such that $\chi(u) = \kappa$. The domain is then $\kappa$ pseudoholomorphic disks, each of which is counted in the map $\mathcal{F}_{0,\kappa=1}$. Given a generator $$\bm{x} = \{x_{1\rho(1)}, \dots, x_{\kappa \rho(\kappa)} \} \in CW(\sqcup_{i=1}^\kappa T_{q_i}^*\Sigma,\sqcup_{i=1}^\kappa N^* \alpha_i)_c,$$ where $x_{i\rho(i)} \in CF(\phi^1_{H_V}(T^*_{q_i}\Sigma),N^*\alpha_{\rho(i)})$, $$\mathcal{F}_0(\bm{x})= \otimes_{i=1}^\kappa \mathcal{F}_{0,\kappa=1}(x_{i\rho(i)}) = \otimes_{i=1}^\kappa [\mathcal{L}(x_{i\rho(i)})] \in \otimes_{i=1}^\kappa BSk_1(\Sigma,q_i,\alpha_{\rho(i)},\ast)_c.$$ 

We construct an inverse $$\mathcal{F}_0^{-1}: BSk_{\kappa}(\Sigma,\bm{q},\bm{\alpha}, \ast)_c|_{\hbar=0} \longrightarrow CW(\sqcup_{i=1}^\kappa T_{q_i}^*\Sigma,\sqcup_{i=1}^\kappa N^* \alpha_i)_c$$ 
given by $$\bm{\gamma} = \{\gamma_{1\rho(1)}, \dots,\gamma_{\kappa\rho(\kappa)}\} \mapsto \{\mathcal{F}_{0,\kappa=1}^{-1}(\gamma_{1\rho(1)}), \dots, \mathcal{F}_{0,\kappa=1}^{-1}(\gamma_{\kappa\rho(\kappa)})\}.$$
Thus $\mathcal{F}_0$ is an isomorphism.
\end{proof}

We use Lemma \ref{lem:hbar0} and reintroduce $\hbar$. It suffices to show that $\mathcal{F}$ is a bijection. We repeat the argument in \cite{honda2022higher}. 

\begin{proof}[Proof of Theorem \ref{thm:Fiso1}]
\hfill\\

    \textit{Injectivity of $\mathcal{F}$}:
    Suppose that there exists $\bm{a} \neq 0$ such that $\mathcal{F}(\bm{a}) = 0$. We can write $\bm{a} = \sum_{i\geq 0}\hbar^i \bm{a}_i$, where $\bm{a}_i \in CW(\sqcup_{i=1}^\kappa T_{q_i}^*\Sigma,\sqcup_{i=1}^\kappa N^* \alpha_i)_c|_{\hbar=0}$. Since the codomain $BSk_\kappa(\Sigma,\bm{q},\bm{\alpha}, \ast)$ has no $\hbar$-torsion, it follows that $\bm{a}_0 \neq 0$. Then setting $\hbar = 0$, we have $\mathcal{F}(\bm{a}_0) =\mathcal{F}(\bm{a}) = 0$. This implies that $\mathcal{F}_0(\bm{a}_0) = 0$, and thus $\bm{a}_0=0$, a contradiction. Therefore, $\mathcal{F}$ is injective. \\

    \textit{Surjectivity of $\mathcal{F}$}: Let $\bm{b} \in BSk_\kappa(\Sigma,\bm{q},\bm{\alpha}, \ast)$. By Lemma \ref{lem:hbar0}, there exists $\bm{a}_0$ such that $\mathcal{F}(\bm{a}_0) \equiv \bm{b} \text{ (mod }  \hbar)$. Let 
    $$\bm{b_1} = \frac{\bm{b}-\mathcal{F}(\bm{a}_0)}{\hbar}|_{\hbar=0}.$$
    Then there exists an $\bm{a}_1$ such that $\mathcal{F}(\bm{a}_1) \equiv \bm{b}_1 \text{ (mod } \hbar)$. Repeating this procedure, we get $\mathcal{F}(\sum_{i\geq 0}\hbar^i \bm{a}_i) = \bm{b}$. Thus $\mathcal{F}$ is surjective. 
\end{proof}

\section{Geometric realization of $CW(\sqcup_{i=1}^\kappa T_{q_i}^*\Sigma,\sqcup_{i=1}^\kappa N^* \alpha_i)_c$ as a module over the braid skein algebra}
\subsection{Algebraic action}
\hfill\\
\label{sec:algaction}We describe the (right) braid skein-module structure on $BSk_\kappa(\Sigma,\bm{q},\bm{\alpha}, \ast)$. Specifically, we give the action of the braid skein algebra $BSk_\kappa(\Sigma, \bm{q}, \ast)$ on our proposed module $BSk_\kappa(\Sigma,\bm{q},\bm{\alpha}, \ast)$. 

Let $[\gamma_1]$ be an element of the braid skein algebra $BSk_\kappa(\Sigma, \bm{q}, \ast)$ and $[\gamma_2]$ be an element of $BSk_\kappa(\Sigma,\bm{q},\bm{\alpha}, \ast)$. Then $[\gamma_1]$ and $ [\gamma_2]$ represent $c$-deformed homotopy classes of paths in $\Omega(\text{UConf}_\kappa((\Sigma \setminus \{ \ast \}), \bm{q})$ and $\Omega(\text{UConf}_\kappa((\Sigma \setminus \{ \ast \}), \bm{q}, \bm{\alpha})$, respectively. Suppose $\gamma_i$ is a representative for $[\gamma_i]$, where $\gamma_i$ is an element of $\Omega(\text{UConf}_\kappa((\Sigma \setminus \ast)) \otimes \Z[c^{\pm 1}]  \otimes \Z[[\hbar]]$ satisfying the appropriate boundary conditions for the configuration space. 

\begin{proposition}
       $BSk_\kappa(\Sigma, \bm{q}, \ast)$ acts on $BSk_\kappa(\Sigma,\bm{q},\bm{\alpha}, \ast)$ by the map 
       $$\rho: BSk_\kappa(\Sigma,\bm{q},\bm{\alpha}, \ast) \otimes BSk_\kappa(\Sigma, \bm{q}, \ast)   \longrightarrow BSk_\kappa(\Sigma,\bm{q},\bm{\alpha}, \ast) $$

\noindent given by 
   $$\rho([\gamma_2]_c,[\gamma_1]_c) \mapsto [\gamma_1 \gamma_2]_c.$$ 

   \begin{proof}
       This follows directly from the product defined on the braid skein algebra. 
   \end{proof}

\end{proposition}

\noindent \textit{Notation}: We will denote the product as a left multiplication to be more compatible with the composition of braids. Specifically, given $\gamma_1 \in BSk_\kappa(\Sigma, \bm{q}, \ast),\  \gamma_2 \in  BSk_\kappa(\Sigma,\bm{q},\bm{\alpha}, \ast)$, we write $\gamma_1 \gamma_2 := \rho(\gamma_2,\gamma_1)$.

\subsection{Geometric action}
\label{sec:geoaction}
\hfill\\
We make use of the bijection between elements in wrapped HDHF and $\kappa$-tuples of perturbed geodesics. Consider the Hamiltonian $H_V(t,q,p) = \frac{1}{2}|p|^2 + V(t,q)$ and its Fenchel dual $L_V(t,q,v)= \frac{1}{2}|v|^2 - V(t,q)$. Then, as in Section \ref{sec:pathspaceandwrappedhdhf}, we consider the Morse function $\mathcal{A}_V$ on the path space $\Omega^{1,2}(\Sigma,q,q')$ defined by $$\mathcal{A}_V(\gamma) = \int_0^1 L_V(t,\gamma(t), \dot{\gamma}(t)) dt.$$  We generalize this to $\kappa$ strands by considering the path spaces
$$\Omega_\rho(\Sigma, \bm{q}, \bm{q}):=\prod_{i=1}^\kappa \Omega(\Sigma,q_i,q_{\rho(i)}),\ \ \  \Omega^{1,2}_\rho(\Sigma, \bm{q}, \bm{q}):=\prod_{i=1}^\kappa \Omega^{1,2}(\Sigma,q_i,q_{\rho(i)}),$$
$$\Omega^{1,2}(\Sigma, \bm{q}, \bm{q}):=\bigsqcup_{\rho \in S_\kappa}\Omega^{1,2}_\rho(\Sigma, \bm{q}, \bm{q}),$$
where $\rho \in S_\kappa$ is a permutation. Then given $\bm{\gamma} \in \Omega^{1,2}_\rho(\Sigma, \bm{q}, \bm{q})$, we define 
$$\mathcal{A}_V(\bm{\gamma}) = \sum_i \mathcal{A}_V(\gamma_i).$$

For generic $V$, the action functional $\mathcal{A}_V$ on $\Omega^{1,2}_\rho(\Sigma, \bm{q}, \bm{q})$ is a Morse function which satisfies the Palais-Smale condition. 

The critical points of $\mathcal{A}_V$ are exactly the $\kappa$-tuples of perturbed geodesics which are in bijection with the elements of our wrapped HDHF. 

We define the path space $\Omega^{1,2}(\Sigma,\bm{q},\bm{\alpha})$ in a similar manner, the only difference being the end points of the paths are confined to $\alpha_i$ instead of $q'_i$.

Let $\bm{y} \in CW( \sqcup_{i=1}^\kappa T^*_{q_i} \Sigma)$ and $\bm{x} \in CW(\sqcup_{i=1}^\kappa T^*_{q_i} \Sigma, \sqcup_{i=1}^\kappa N^* \alpha_i)$. We want to define the product $\bm{y} \cdot \bm{x}$. On one hand, this has already been done in the framework of HDHF (\cite{colin2020applications}), where a $\mu^2$ map gives us the product $\bm{y} \cdot \bm{x}=\mu^2(\bm{x},\bm{y}) \in CW(\sqcup_{i=1}^\kappa T^*_{q_i} \Sigma, \sqcup_{i=1}^\kappa N^* \alpha_i)$. We aim to give a more geometrically intuitive interpretation of the product using the bijection with perturbed geodesics, making it more compatible with the algebraic action given in Section \ref{sec:algaction}.

We define a similar Morse function where the endpoint of the curve is allowed to move along our curves $\alpha_1,\dots,\alpha_\kappa$. With this setup, the geometric action of $CW(\sqcup_{i=1}^\kappa T^*_{q_i}\Sigma)_c$ on $CW(\sqcup_{i=1}^\kappa T_{q_i}^*\Sigma,\sqcup_{i=1}^\kappa N^* \alpha_i)_c$ is given by translating elements to $V$-perturbed geodesics viewed as paths, concatenating the paths, then performing the Morse gradient flow discussed above. The result is a $V$-perturbed geodesic from $\sqcup_{i=1}^\kappa q_i$ to $\sqcup_{i=1}^\kappa \alpha_i$, which is then identified with an element of $CW(\sqcup_{i=1}^\kappa T_{q_i}^*\Sigma,\sqcup_{i=1}^\kappa N^* \alpha_i)_c$. 

When performing the Morse gradient flow after concatenation, we must consider passings through the marked point as well as any crossings of the perturbed geodesics. Similarly as before, to account for the marked point, we have the homotopy $H$ from the concatenated curves $\gamma_1 \gamma_2$ to the resulting geodesic $\gamma$ and we impose on this the $c$-deformed homotopy relation $[\gamma] = c^{2\langle H,\ast \rangle}[\gamma_1 \gamma_2]$. 
We account for any crossings that occur on $\Sigma$ during our flow by applying the HOMFLY skein relation at all crossings. We can view each $\kappa$-tuple of perturbed geodesics as a braid in $[0,1] \times \Sigma$ by mapping $\gamma(t) \mapsto (t,\gamma(t))$. Then each crossing will either be a positive crossing $\sigma_i$ or a negative crossing $\sigma_i^{-1}$, and our HOMFLY skein relation is $\sigma_i - \sigma_i^{-1} = \hbar e$, where $e$ is the resolution of the crossing. We call any time a positive crossing changes to a negative crossing (or vice versa) a switching. In the Morse theory view, this results in a bifurcated trajectory; one trajectory is a continuation of the switching and the other is a resolution of the crossing with a factor of $\hbar$. As a braid, we get a sum of the same braid with the crossing reversed and a resolved crossing with an extra $\hbar$ factor. Using this relation, we can resolve all of the crossings into $\Z[\hbar,c^{\pm 1}]$-linear combinations of braids which are tuples of $V$-perturbed geodesics.

\subsection{Equivalence of DAHA modules}
\hfill\\
Our main goal of this section is to show that the following diagram commutes:
\begin{equation}
\begin{tikzcd}
\label{eq:diagram}
 CW(\sqcup_i T^*_{q_i} \Sigma, \sqcup_i N^* \alpha_i)_c \otimes CW(\sqcup_i T^*_{q_i} \Sigma)_c \arrow[rr,"\mathcal{F}_1 \otimes \mathcal{F}_2"] \arrow[d, "\mu^2"]
&& BSk_\kappa(\Sigma,\bm{q},\bm{\alpha}, \ast) \otimes (H_\kappa(\Sigma, \bm{q}) \otimes \Z[[\hbar]]) \arrow[d, "\rho" ] \\
CW(\sqcup_i T^*_{q_i} \Sigma, \sqcup_i N^* \alpha_i)_c \arrow[rr, "\mathcal{F}_1"]
&&  BSk_\kappa(\Sigma,\bm{q},\bm{\alpha}, \ast)
\end{tikzcd}\end{equation}

We recall each map of the diagram:

\noindent The top maps $$\mathcal{F}_1: CW(\sqcup_{i=1}^\kappa T^*_{q_i} \Sigma, \sqcup_{i=1}^\kappa N^* \alpha_i)_c \longrightarrow BSk_\kappa(\Sigma,\bm{q},\bm{\alpha}, \ast)$$ and $$\mathcal{F}_2: CW(\sqcup_{i=1}^\kappa T^*_{q_i} \Sigma)_c \longrightarrow H_\kappa(\Sigma, \bm{q}) \otimes \Z[[\hbar]] $$
are the isomorphisms defined in Sections \ref{sec:parameterc} and \ref{sec:HeckeHDHF}, respectively. 

\noindent The left map $$\mu^2: CW(\sqcup_{i=1}^\kappa T^*_{q_i} \Sigma, \sqcup_{i=1}^\kappa N^* \alpha_i)_c \otimes CW(\sqcup_{i=1}^\kappa T^*_{q_i} \Sigma)_c  \longrightarrow CW(\sqcup_{i=1}^\kappa T^*_{q_i} \Sigma, \sqcup_{i=1}^\kappa N^* \alpha_i)_c$$ is the $A_{\infty}$-map defined by Equation \eqref{highermaps} with $m=2$. It is given by a count of holomorphic maps which project onto a thrice-punctured disk satisfying the usual boundary conditions; see \cite{colin2020applications} for more details.

\noindent The right map $$\rho: BSk_\kappa(\Sigma,\bm{q},\bm{\alpha}, \ast) \otimes (H_\kappa(\Sigma, \bm{q}) \otimes \Z[[\hbar]]) \longrightarrow BSk_\kappa(\Sigma,\bm{q},\bm{\alpha}, \ast) $$
is given in Section \ref{sec:algaction} after identifying $H_\kappa(\Sigma, \bm{q}) \otimes \Z[[\hbar]]$ with $BSk_\kappa(\Sigma, \bm{q}, \ast)$ . 

\noindent The bottom map is again the evaluation map from Section \ref{sec:parameterc}. 
\begin{lemma}
Diagram 5.1 commutes. 
\end{lemma}Before we prove the lemma, we introduce a moduli space of holomorphic curves similar to that of \cite[Section 6.2]{honda2022higher}.

Let $T_2 := D_3$ be our $A_\infty$-base where $\partial_iT_2 = \partial_i D_3$. Let $\mathcal{T}_2$ be the moduli space of $T_2$ modulo automorphisms, and choose representatives $T_2$ of equivalence classes in a smooth manner. Let $\pi_{T^*\Sigma}$ be the projection $T_2 \times T^*\Sigma \to T^*\Sigma$ and choose a sufficiently generic consistent collection of compatible almost complex structures such that they are close to a split almost complex structure projecting holomorphically to $T_2$, as in Section \ref{sec:HDHFreview}. Perturb the 0-section near the $\alpha_i$ and let $\bm{x}$ be the tuple of intersections $\alpha_i \cap \phi_{H'}(\Sigma)$ corresponding to the bottom generators, where $\phi_{H'}(\Sigma)$ is the perturbed 0-section. 
We denote by $\mathcal{H}(\bm{q}'',\bm{y}',\bm{y},\bm{x})$ the moduli space of maps $$u: (\dot{F}, j) \longrightarrow (T_2 \times T^*\Sigma, J_{T_2}),$$ where $(F,j)$ is a compact Riemann surface with boundary, $\bm{p}_0$, $\bm{p}_1$, $\bm{p}_2$, $\bm{p}_3$ are disjoint tuples of boundary punctures of $F$ and $\dot{F} = F \setminus \cup_i \  \bm{p}_i$, satisfying:

  \[
    \begin{cases}
    du \circ j = J_{T_2} \circ du;\\

    \pi_{T^*\Sigma} \circ u(z) \in \phi^2_{H_V}(\sqcup_{i=1}^\kappa T^*_{q_i}\Sigma)\text{ if } \pi_{T_2} \circ u(z) \subset \partial_0 T_2;\\
     \text{each component of } \partial \dot{F} \text{ that projects to }\partial_0T_2  \text{ maps to a distinct } \phi^2_{H_V}(T^*_{q_i}\Sigma);\\
    
   \pi_{T^*\Sigma} \circ u(z) \in \phi^1_{H_V}(\sqcup_{i=1}^\kappa T^*_{q_i}\Sigma) \text{ if } \pi_{T_2} \circ u(z) \subset \partial_1 T_2;\\
    \text{each component of } \partial \dot{F} \text{ that projects to }\partial_1T_2 \text{ maps to a distinct } \phi^1_{H_V}(T^*_{q_i}\Sigma);\\

   \pi_{T^*\Sigma} \circ u(z) \in \sqcup_{i=1}^\kappa N^*\alpha_i \text{ if } \pi_{T_2} \circ u(z) \subset \partial_2 T_2;\\
     \text{each component of } \partial \dot{F} \text{ that projects to } \partial_2 T_2 \text{ maps to a distinct }N^*\alpha_i;\\

   \pi_{T^*\Sigma} \circ u(z) \in \Sigma \times \{0\} \subset T^*\Sigma \text{ if } \pi_{T_2} \circ u(z) \subset \partial_3 T_2;\\

    \pi_{T^*\Sigma} \circ u \text{ tends to }\bm{q}'',\ \bm{y}',\  \bm{y},\  \bm{x}\text{ as }s_0,s_1 ,s_2,s_3 \to +\infty;\\

    \pi_{T_1} \circ u\text{ is a }\kappa\text{-fold branched cover of a fixed }T_2\in \mathcal{T}_2.\\
    
    \end{cases}
    \]

In simpler terms, we look at the moduli space of holomorphic curves between the Lagrangians involved and the zero section of $T^* \Sigma$ in the framework of HDHF.

\begin{figure}[h]
    \centering
    \includegraphics[width=0.6\textwidth]{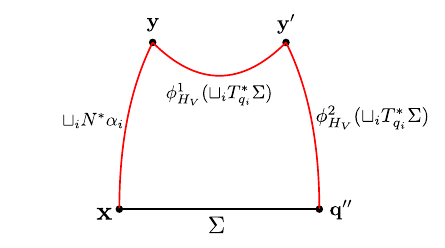}
    \caption{The $A_\infty$-base $T_2$ with boundary conditions}
    
    \label{fig:t2}
\end{figure}

\begin{lemma}
    There exists a sufficiently generic consistent collection of almost complex structures such that the moduli space $\mathcal{H}(\bm{q}'',\bm{y}',\bm{y},\bm{x})$ is of dimension 1 and is transversely cut out for all $\bm{x},\ \bm{y},\ \bm{y}'$ and $\bm{q''}$. Moreover, $\mathcal{H}(\bm{q}'',\bm{y}',\bm{y},\bm{x})$ admits a compactification $\overline{\mathcal{H}}(\bm{q}'',\bm{y}',\bm{y},\bm{x})$ such that its boundary $\partial \overline{\mathcal{H}}(\bm{q}'',\bm{y}',\bm{y},\bm{x})$ is of dimension 0 and contains discrete broken or nodal curves. 
\end{lemma}

\begin{proof}
This is identical to Lemma 6.4 in \cite{honda2022higher}. 
\end{proof}

As in Section \ref{sec:initialT1}, we define a map from $\mathcal{H}(\bm{q}'',\bm{y}',\bm{y},\bm{x}) \times [0,1] \longrightarrow (\Sigma)^\kappa$ by 
\begin{equation}
    \gamma(u)(t) = (\pi_{T^* \Sigma} \circ u) \circ (\pi_{\Sigma} \circ u)^{-1} \circ \tau(t),
\end{equation}
where $\tau:[0,1] \to \partial_3 T_2$ parametrizes the boundary arc from $p_0$ to $p_3$. 
Let $$\mathcal{H}_0(\bm{q}'',\bm{y}',\bm{y},\bm{x}) = \{u \in \mathcal{H}(\bm{q}'',\bm{y}',\bm{y},\bm{x}) ~|~ \gamma(u)(t) \in \text{UConf}_\kappa(\Sigma \setminus \{ \ast\}) \text{ for all } t\}.$$

As before, we define the evaluation map 
$$\mathcal{G}: \mathcal{H}_0(\bm{q'',y',y,x}) \longrightarrow BSk_\kappa(\Sigma,\bm{q},\bm{\alpha}, \ast),$$
$$u \mapsto (-1)^{\natural(u)} \cdot c^{2\langle u, \ast \rangle} \cdot  \hbar^{\kappa - \chi(u)} \cdot [\gamma(u)].$$

\begin{proof}[Proof of Lemma 5.2]
    We analyze the boundary of the index 1 moduli space $\overline{\mathcal{H}}(\bm{q}'',\bm{y}',\bm{y},\bm{x})$ by considering the possible degenerations. Let $\overline{\mathcal{H}}^\chi(\bm{q}'',\bm{y}',\bm{y},\bm{x})$ be the subset of $\overline{\mathcal{H}}(\bm{q}'',\bm{y}',\bm{y},\bm{x})$ consisting of maps with $\chi(u) = \chi$.
    For a generic $u$, $u \in \mathcal{H}_0(\bm{q}'',\bm{y}',\bm{y},\bm{x})$. However, for a 1-parameter family $u_t \in \mathcal{H}(\bm{q}'',\bm{y}',\bm{y},\bm{x})$, $\gamma(u_t)$ may intersect the marked point $\ast$ at some $t \in (0,1)$. However, since we are taking $c$-deformed homotopy classes, we are guaranteed that $\mathcal{G}(u_0) = \mathcal{G}(u_1)$. Hence we will not worry about intersections with $\ast$. 
    All codimension-1 degenerations occur in the $A_\infty$-base direction, giving us a nice characterization of the possible breakings.
    The three types of boundary degenerations are:
    \begin{enumerate}
        \item $\bigsqcup_{\bm{y}'',\chi' + \chi'' - \kappa = \chi}\mathcal{M}^{\text{ind}=0,\chi'}(\bm{y}', \ \bm{y},\ \bm{y}'') \times \mathcal{H}^{\text{ind}=0,\chi''}(\bm{q}'', \ \bm{y}'',\ \bm{x})$;
        \item $\bigsqcup_{\bm{q}',\chi' + \chi'' - \kappa = \chi}\mathcal{H}^{\text{ind}=0,\chi'}(\bm{q''}, \ \bm{y}',\ \bm{q}') \times \mathcal{H}^{\text{ind}=0,\chi''}(\bm{q}', \ \bm{y},\ \bm{x})$;
        \item the set $\partial_n \overline{\mathcal{H}}^{\text{ind}=1,\chi}(\bm{q}'',\bm{y}',\bm{y},\bm{x})$ with a nodal degeneration along $\Sigma$. 
    \end{enumerate}

The first type is shown on the left-hand side of Figure \ref{fig:T2degeneration} and contributes $\mathcal{F}_1(\mu^2(\bm{y,\ y'}))$.

The second type is shown on the right-hand side of Figure \ref{fig:T2degeneration} and contributes $\rho((\mathcal{F}_1 \otimes \mathcal{F}_2)(\bm{y}, \ \bm{y}'))$. 

In fact, all contributions to $\mathcal{F}_1(\mu^2(\bm{y},\ \bm{y}'))$ and $\rho((\mathcal{F}_1 \otimes \mathcal{F}_2)(\bm{y}, \ \bm{y}'))$ come from such degenerations. 

The proof of Proposition 6.5 in \cite{honda2022higher} shows that the total contribution of the third type over all Euler characteristics $\chi$ is zero. Hence it follows that $\mathcal{F}_1(\mu^2(\bm{y},\ \bm{y}')) = \rho((\mathcal{F}_1 \otimes \mathcal{F}_2)(\bm{y}, \ \bm{y}'))$ and so the diagram commutes.

\end{proof}

\begin{figure}[h]
    \centering
    \includegraphics[width=0.8\textwidth]{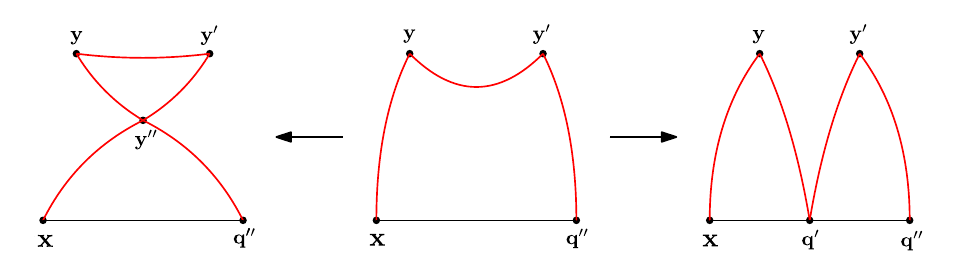}
    \caption{$T_2$ degenerations}
    
    \label{fig:T2degeneration}
\end{figure}

We have thus shown that the map $\mathcal{F}_1: CW(\sqcup_{i=1}^\kappa T^*_{q_i} \Sigma, \sqcup_{i=1}^\kappa N^* \alpha_i)_c \longrightarrow BSk_\kappa(\Sigma,\bm{q},\bm{\alpha}, \ast)$ realizes the wrapped HDHF of $\kappa$ cotangent fibers and $\kappa$ conormal bundles of simple closed curves as a DAHA-module. 

\section{The enhanced polynomial representation}
\label{sec:enhanced poly rep}
In this section, we specialize to $\Sigma = T^2$. We introduce the double affine Hecke algebra (DAHA) along with its polynomial representation. After fixing a configuration of points $q_1,\dots, q_\kappa$ and curves $\alpha_1, \dots, \alpha_\kappa$ in $T^2$, we define the enhanced polynomial representation of DAHA on $CW(\sqcup_{i=1}^\kappa T_{q_i}^*T^2,\sqcup_{i=1}^\kappa N^* \alpha_i)_c$ and prove Theorem \ref{thm:mainaction}.  

\subsection{Double affine Hecke algebra and its polynomial representation}
\label{sec:DAHApolyrep}
\hfill\\
We briefly review the DAHA and its skein-theoretic realization using braids in the punctured torus. For more details, refer to \cite{morton2021dahas}, where these results are proven and discussed at length.

Viewing $T^2$ as a square $I \times I$ with opposite sides identified, we choose $\ast = (\frac{1}{2},\frac{1}{2})$ and let the $\kappa$ points $q_1,\dots,q_\kappa$ line up in increasing fashion along the lower part of the diagonal from $(0,0)$ to $\ast$. We choose a convenient basis for the braids in $B_\kappa(T^2 \setminus \{ \ast \},\bm{q})$. Let $x_i \ (\text{respectively, } y_i)$ be the braid which consists of the point $q_i$ moving uniformly around the $(-1,0)\ (\text{respectively, } (0,1))$ curve. Let $\sigma_i$ for $1\leq i\leq \kappa-1$ be the braid which locally exchanges the strings from $q_i$ and $q_{i+1}$ in a counterclockwise direction when looking down onto $T^2$, as shown in Figure 7.1(B) below. 

\begin{figure}[h]
    \begin{subfigure}{.5\textwidth}
    \centering
    \includegraphics[width=1\textwidth]{Figures/xybraids.png}
    \caption{Braids $x_i$ and $y_i$ }
    \label{fig:xybraids}
    \end{subfigure}%
\begin{subfigure}{.5\textwidth}
    \centering
    \includegraphics[width=0.65\textwidth]{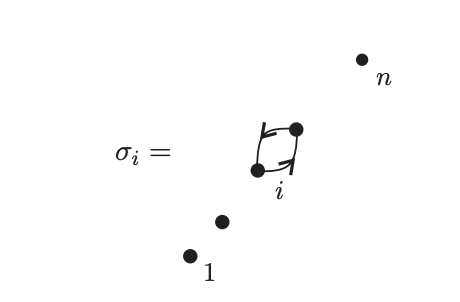}
    \caption{Braid $\sigma_i$}
    \label{fig:sigmabraid}
\end{subfigure}
\caption{Generators for the braid group on the punctured torus}

\end{figure}

The key element in this geometric realization of DAHA is the next theorem due to Morton and Samuelson.

\begin{theorem}[Morton-Samuelson]
\label{thm:DAHAiso}
The braid skein algebra $BSk_\kappa(T^2,\bm{q},\ast)$ is isomorphic to the double affine Hecke algebra $\ddot{H}_\kappa$.
    
\end{theorem}

We now fix the presentation for the skein algebra, and therefore the DAHA, that we will be using.

\begin{theorem}
\label{DAHAbasis}
The double affine Hecke algebra $\ddot{H}_\kappa$ can be presented by the braids $\sigma_1,\dots,\sigma_{\kappa-1},x_1,y_1$
with relations:    
\begin{enumerate}
\item $\sigma_i\sigma_j = \sigma_j \sigma_i, \ |i-j|>1,$
\item $\sigma_i\sigma_{i+1}\sigma_i = \sigma_{i+1}\sigma_i\sigma_{i+1},$
\item $\sigma_ix_1 = x_1 \sigma_i, \ i>1,$
\item $\sigma_iy_1 = y_1\sigma_i, \ i>1,$
\item $x_1\sigma_1x_1\sigma_1 = \sigma_1x_1\sigma_1x_1,$
\item $y_1\sigma_1y_1\sigma_1 = \sigma_1y_1\sigma_1y_1,$
\item $x_1\sigma_1y_1\sigma_1^{-1} = \sigma_1y_1\sigma_1x_1,$
\item $(\sigma_1-s)(\sigma_1 + s^{-1}) = 0,$
\item $x_1^{-1}y_1x_1y_1^{-1} = c^2\sigma_1\sigma_2\cdot\cdot\cdot \sigma_{\kappa-1}\sigma_{\kappa-1}\cdot \cdot \cdot \sigma_2\sigma_1$.

\end{enumerate}

\end{theorem}

\begin{remark}
    The relations stated above are slightly different than those in \cite{morton2021dahas}. Specifically, we replace $x_1$ with $x_1^{-1}$ since the generator $x_i$ in \cite{morton2021dahas} corresponds to a loop based at $q_1$ in the $(1,0)$ direction whereas our generator goes in the $(-1,0)$ direction. The relations have been adjusted with this in mind. 
\end{remark}

Although $\ddot{H}_\kappa$ can be generated by $\sigma_1,\cdots, \sigma_{\kappa-1},\ x_1, $ and $y_1$, it will be convenient to make explicit the expression for the braids $x_i$ and $y_i$. Using the relations $\sigma_i x_i \sigma_i = x_{i+1}$ and $\sigma_i y_i \sigma_i = y_{i+1}$, it follows that $x_i = \sigma_{i-1}\cdots \sigma_{1}x_1 \sigma_1\cdots \sigma_{i-1}$ and $y_i =  \sigma_{i-1}\cdots \sigma_{1}y_1 \sigma_1\cdots \sigma_{i-1}$. \\

The DAHA $\ddot{H}_\kappa$ has a representation on $\Z[[s]][c^{\pm 1}][X_1^{\pm 1}, \cdots, X_\kappa^{\pm 1}]$ called the {\em polynomial representation}. Note that in this paper we are concerned with the polynomial representation of the DAHA $\ddot{H}_\kappa$ instead of the more common spherical DAHA; we will use the presentation given in \cite{Enomoto2009}.

\begin{definition}
\label{def:polyrep}
    The {\em polynomial representation} of $\ddot{H}_\kappa$ on $\Z[[s]][c^{\pm 1}][X_1^{\pm 1}, \cdots, X_\kappa^{\pm 1}]$ is defined by the following:
    \begin{align}
      x_i \notag &\mapsto X_i,\\
    \sigma_i \notag &\mapsto s\tau_i + \frac{s-s^{-1}}{X_iX_{i+1}^{-1} - 1}(\tau_i - 1),\\
    y_1 \notag &\mapsto \sigma_1^{-1}\cdots \sigma_{\kappa-1}^{-1}\omega,  
    \end{align}

    \noindent where $\tau_i$ permutes $X_i$ and $X_{i+1}$ and for any $f\in \Z[[s]][c^{\pm 1}][X_1^{\pm 1}, \cdots, X_\kappa^{\pm 1}]$, $$(\omega f)(X_1,\cdots,X_\kappa) = f(c^2 X_\kappa,X_1,\cdots,X_{\kappa-1}).$$

   Denote the action above by $$p: \ddot{H}_\kappa \times \Z[[s]][c^{\pm 1}][X_1^{\pm 1}, \cdots, X_\kappa^{\pm 1}] \longrightarrow \Z[[s]][c^{\pm 1}][X_1^{\pm 1}, \cdots, X_\kappa^{\pm 1}].$$

\end{definition}

\begin{remark}
    Observe that in the definition of the DAHA, the variable $s$ does not appear on its own. The HOMFLY skein relation in the definition of the braid skein algebra uses $s-s^{-1}$; in the presentation of Theorem \ref{DAHAbasis}, expanding the relation (8) gives the term $s-s^{-1}$.
\end{remark}

For our purposes, we let $\hbar = s-s^{-1}$ and change the coefficient ring from $\Z[[s]][c^{\pm 1}]$ to $\Z[[\hbar]][c^{\pm 1}]$ when we are not dealing with the polynomial representation.

\subsection{The enhanced polynomial representation}
\hfill\\
In this section, we compute the DAHA-module $BSk_\kappa(T^2,\bm{q},\bm{\alpha}, \ast)$ for the configuration of points $q_1, \dots, q_\kappa$ as in Section \ref{sec:DAHApolyrep} and simple closed curves $\alpha_1,\dots,\alpha_\kappa$ similar to Section \ref{modelcalc}. 

Viewing $T^2$ as $I \times I$ with opposite sides identified, let $q_i = (\frac{i}{2(\kappa+1)} ,\frac{i}{2(\kappa+1)})$ and $\alpha_i = \{\frac{1}{2} + \frac{i}{2\kappa+2}\} \times I$. Moreover, let $\ast = (\frac{1}{2},\frac{1}{2})$. Choose a perturbation term $V(t,q)$ such that $q'_i = \phi^1_{H_V}(T^*_{q_i}T^2) \cap T^2$ is to the left of $q_i$ when viewed on $I \times I$ and $|q_i - q'_i|>|q_j-q'_j|$ whenever $i<j$.

We introduce an additional parameter $d$ which keeps track of the ends of the braids sliding along the $\alpha_i$. 
Specifically, consider the projection $$\pi_\alpha: \alpha_1 \times \cdots \times \alpha_\kappa \longrightarrow T^\kappa$$ of the $\alpha_i$ to $T^\kappa = (S^1)^\kappa$ given by dropping the $x$ coordinate for each $\alpha_i$. 
Let $$\Delta = \{(x_1, \dots, x_\kappa) ~|~ x_i = x_j \text{ for some } i\neq j \}$$ be the big diagonal in $T^\kappa$. The parameter $d$ counts signed intersections with $\Delta$ as the braids are isotoped. 

\begin{definition}
\label{def:d-deformed}
    The {\em $d$-deformed braid skein group} $BSk_\kappa(T^2,\bm{q},\bm{\alpha},\ast)_d$ is the free $\Z[[\hbar]][c^{\pm 1},d^{\pm 1}]$-module generated by elements of the braid skein group $BSk_\kappa(T^2,\bm{q},\bm{\alpha},\ast)$ subject to the {\em ends-slide relation}:

\begin{figure}[h!]
\label{fig:sliderelation}
    \centering
    \includegraphics[width=0.4\textwidth]{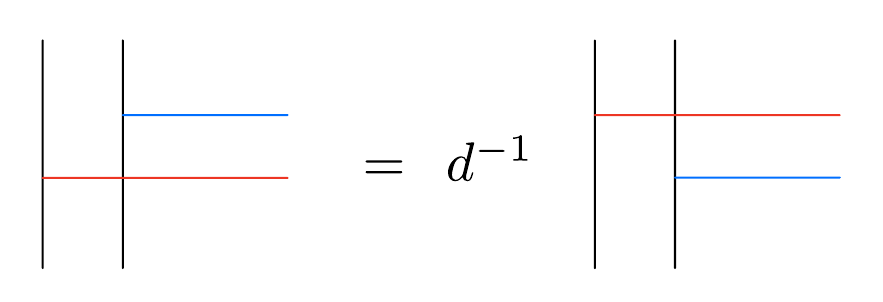}

\end{figure}

    Let $[\gamma]_{c,d}$ be the equivalence class of the braid $\gamma$ in $BSk_\kappa(T^2,\bm{q},\bm{\alpha},\ast)_d$ which keeps track of the parameters $c$ and $d$.
\end{definition}

We now expand our wrapped HDHF group with parameter $c$ to one with parameters $c$ and $d$. 
\begin{definition}
    The {\em wrapped HDHF group of $\sqcup_{i=1}^\kappa T_{q_i}^*T^2$ and $\sqcup_{i=1}^\kappa N^* \alpha_i$ with parameters $c$ and $d$} is $$CW(\sqcup_{i=1}^\kappa T_{q_i}^*T^2,\sqcup_{i=1}^\kappa N^* \alpha_i)_{c,d} := CW(\sqcup_{i=1}^\kappa T_{q_i}^*T^2,\sqcup_{i=1}^\kappa N^* \alpha_i)_c \otimes \Z[d^{\pm 1}].$$
\end{definition}
A slight enhancement of Theorem \ref{thm:Fiso1} which keeps track of intersections with $\Delta \subset T^\kappa$ viewed as a subset of $T^2\times\{1\}$ gives:\\

\noindent\textbf{Theorem 1.1'.}
    $CW(\sqcup_{i=1}^\kappa T_{q_i}^*T^2,\sqcup_{i=1}^\kappa N^* \alpha_i)_{c,d} \simeq BSk_\kappa(T^2,\bm{q},\bm{\alpha},\ast)_d$\\

Under the configuration of points $q_i$ and curves $\alpha_i$ described at the start of this subsection, we choose a specific presentation for $BSk_\kappa(T^2,\bm{q},\bm{\alpha}, \ast)_d$ (and $CW(\sqcup_{i=1}^\kappa T_{q_i}^*T^2,\sqcup_{i=1}^\kappa N^* \alpha_i)_{c,d}$). 

Consider a $\kappa$-tuple of perturbed geodesics $\bm{\gamma}=\{\gamma_1,\dots, \gamma_\kappa\}$ viewed as a braid in $BSk_\kappa(T^2,\bm{q},\bm{\alpha}, \ast)_d$, where each $\gamma_i$ is a perturbed geodesic from $q_i$ to $\alpha_{\sigma(i)}$ for some permutation $\sigma \in S_\kappa$. Let $\alpha = \{ \frac{1}{2}\} \times S^1 \in T^2$ be a simple closed curve between the set of points $q_i$ and curves $\alpha_i$. We define the signed intersection number $n_i = \langle \gamma_i, \alpha \rangle$ of each perturbed geodesic with $\alpha$. The sign of $n_i$ is set to be positive if $\gamma_i$ is in the $(-1,0)$ direction and negative if it is in the $(1,0)$ direction. This intersection number is similar to the process described in the proof of Lemma \ref{lem:modelcomp}. 

The following is a slight enhancement of Lemma \ref{lemma:presentation}:\\

\noindent\textbf{Lemma 1.2'.}
    $BSk_\kappa(T^2,\bm{q},\bm{\alpha}, \ast)_d \simeq (\Z[a_1^{\pm 1}, \dots, a_{\kappa}^{\pm 1}] \otimes \Z[S_{\kappa}])  \otimes \Z[c^{\pm 1}] \otimes \Z[[\hbar]] \otimes \Z[d^{\pm 1}].$ \\
    
    We denote the module with this presentation by $\module$.

\begin{proof}
    Every element in $BSk_\kappa(T^2,\bm{q},\bm{\alpha}, \ast)_d$ can be identified with a $\Z[[\hbar]][c^{\pm 1},d^{\pm 1}]$-linear combination of $\kappa$-tuples of perturbed geodesics viewed as a braid. This is done by homotoping the strands to perturbed geodesics while keeping track of intersections within the strands, with the marked point, and with the big diagonal $\Delta$ of $T^\kappa$. 

   Let $$f: BSk_\kappa(T^2,\bm{q},\bm{\alpha}, \ast)_d \longrightarrow (\Z[a_1^{\pm 1}, \dots, a_{\kappa}^{\pm 1}] \otimes \Z[S_{\kappa}])  \otimes \Z[c^{\pm 1}] \otimes \Z[[\hbar]] \otimes \Z[d^{\pm 1}]$$ be the $\Z[[\hbar]][c^{\pm 1},d^{\pm 1}]$-linear map which sends  
   $$\bm{\gamma} = \{\gamma_1,\dots,\gamma_\kappa \}\mapsto (a_1^{n_1} \cdots a_\kappa^{n_\kappa}, \sigma),$$
   where $\bm{\gamma}$ is a $\kappa$-tuple of perturbed geodesics and $n_i = \langle \gamma_i, \alpha \rangle$ is the signed intersection number described above. 

   A generalization of the model computation in Section \ref{modelcalc} shows us that $f$ is surjective. That is, we can construct a $\kappa$-tuple of perturbed geodesics with the right permutation and intersections with $\alpha$. 
   
   Let $\bm{y} \in BSk_\kappa(T^2,\bm{q},\bm{\alpha}, \ast)_d$ and suppose $f(\bm{y}) = 0$. By an argument similar to the discussion at the end of Section \ref{sec:geoaction}, we can identify $\bm{y}$ with a $\Z[[\hbar]][c^{\pm 1}, d^{\pm 1}]$-linear combination of distinct $\kappa$-tuples of perturbed geodesics $\bm{\gamma}_i$: $$\bm{y} = \sum_{i=1}^n g_i \bm{\gamma}_i,$$ where $g_i \in \Z[[\hbar]][c^{\pm 1}, d^{\pm 1}]$. Since $f(\bm{y}) =0$, it follows that $\sum_{i=1}^n g_i f(\bm{\gamma}_i) = 0$. Since the perturbed geodesics $\bm{\gamma}_i$ are distinct, they belong to different homotopy classes and so their images under $f$ are linearly independent. It follows that $\bm{y} = 0$ and thus $f$ is injective. 
\end{proof}

Consider an element $( 1,\sigma) \in \module$. This element is represented by $V$-perturbed geod\-esics from each $q_i$ to $\alpha_{\sigma(i)}$ which do not intersect each other or the curve $\alpha$. Let $\sigma_i \in \ddot{H}_\kappa$, viewed as a braid consisting of strands $q_i \mapsto q'_{i+1}$, $q_{i+1} \mapsto q'_i$ and $q_j \mapsto q'_j$ for $j \neq i,\ i+1$, where by $a\mapsto b$ we mean "from $a$ to $b$". Our choice of perturbation term $V(t,q)$ guarantees that the strand from $q_i$ crosses over the strand from $q_{i+1}$ when projected down from $T^2 \times [0,1] \to T^2$. 

\begin{lemma}
\label{geodesicisotoping}
    In the situation above, if $\sigma(i) < \sigma(i+1)$, then the concatenation of the braid $\sigma_i$ and the geodesics representing $(1,\sigma)$ can be isotoped to geodesics $d^{-1}(1, \sigma_i \sigma)$. (In the expression $\sigma_i\sigma$, $\sigma_i$ is viewed as an element of $S_\kappa$ under the projection $B_\kappa(T^2,\bm{q})\to S_\kappa$.)  

    On the other hand, if $\sigma(i) > \sigma(i+1)$, then the concatenation of the braid $\sigma_i$ and the geodesics representing $(1,\sigma)$ is equivalent to a linear combination of geodesics $\hbar (1, \sigma) + d(1,\sigma_i\sigma)$. 
\end{lemma}

\begin{proof}
    Suppose $\sigma(i)<\sigma(i+1)$. (See the left-hand side of Figure \ref{fig:sigmaicross}.) Slide the end of the strand going to $\alpha_{\sigma(i+1)}$ along $\alpha_{\sigma(i+1)}$ past the strand going to $\alpha_{\sigma(i)}$. Due to the arrangement of the $\alpha_j$, this creates a crossing in which the sliding strand crosses over the other, picking up a factor of $d^{-1}$. The result is a braid which has a positive crossing at the bottom (by this we mean at a lower $t$-coordinate where the braid is in $T^2\times[0,1]$ with coordinates $(q,t)$) due to the $\sigma_i$ and a negative crossing at the top. Thus we can isotope the two strands apart, arriving at the set of geodesics representing $d^{-1}(1,\sigma_i \sigma)$.

    Suppose on the other hand that $\sigma(i) > \sigma (i+1)$. (See the right-hand side of Figure \ref{fig:sigmaicross}.) Resolve the $\sigma_i$ braid as $\sigma_i^{-1} + \hbar$ to get two braids: $\sigma_i ^{-1} \cdot (1,\sigma) + \hbar (1,\sigma)$. For the $\sigma_i ^{-1} \cdot (1,\sigma)$ braid, slide the strand along $\alpha_{\sigma(i)}$. This creates a positive crossing and picks up a factor of $d$. With this positive crossing, we can isotope the strands apart to get a braid $d(1,\sigma_i\sigma)$. Therefore, $\rho((1,\sigma),\sigma_i) = \hbar (1, \sigma) + d(1,\sigma_i\sigma)$.
\end{proof}

\begin{figure}[h]
    \centering
    \includegraphics[width=0.7\textwidth]{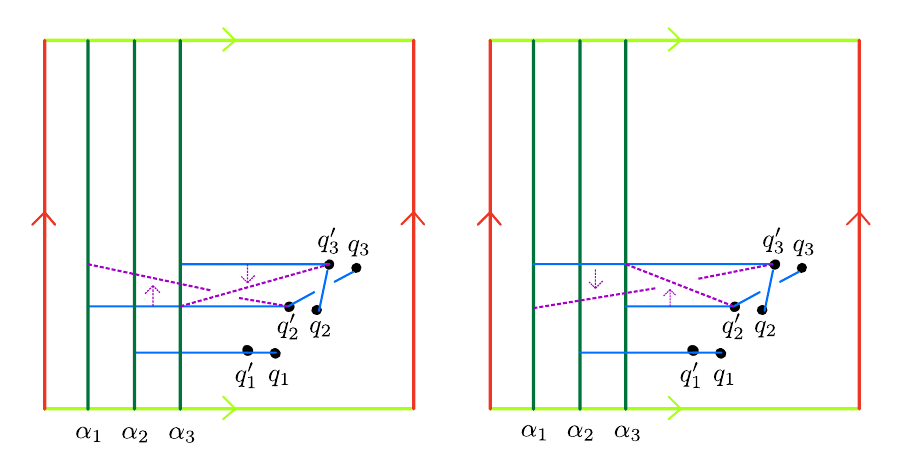}
    \caption{The compositions $\sigma_2 (1,(12))$ and $\sigma_2 (1,(123))$. On the left, we have the case where $\sigma(i)<\sigma(i+1)$ and sliding the ends of the strands creates a crossing of the dotted purple arcs which allows us to separate the strands. On the right, we have the case where $\sigma(i) > \sigma(i+1)$ and we see that the strands are linked after sliding the ends of the strands across each other. }
    
    \label{fig:sigmaicross}
\end{figure}

\noindent \textit{Notation}: Recall that $CW(\sqcup_{i=1}^\kappa T_{q_i}^*T^2,\sqcup_{i=1}^\kappa N^* \alpha_i)_c$ is a right DAHA-module. It follows that there is a right $\ddot{H}_\kappa$ action on $\module$. It is often convenient to identify elements of $\module$ with $\kappa$-tuples of paths which are elements of $BSk_\kappa(T^2,\bm{q},\bm{\alpha}, \ast)$. Since composition of paths is written as a left multiplication, we choose to adopt the notation for a left module. Let $h \in \ddot{H}_\kappa$ and $x \in \module$, then we denote $\rho(x,h)= h \cdot x$, where $\rho$ is as in Section \ref{sec:algaction}. After identifying $h_1,h_2 \in \ddot{H}_\kappa$ with braids $\gamma_1,\gamma_2$ we have that $\rho(x,\tilde{\rho}(h_2,h_1)) = (\gamma_1\gamma_2)\cdot x = \gamma_1 \cdot (\gamma_2 \cdot x)$, a left module structure, where $\tilde{\rho}: BSk_\kappa(T^2,\bm{q}, \bm{q}, \ast) \otimes BSk_\kappa(T^2,\bm{q}, \bm{q}, \ast) \longrightarrow BSk_\kappa(T^2,\bm{q}, \bm{q}, \ast)$ is the product on the braid skein algebra.

\begin{corollary}
\label{thm:averaging}
    Let $\sigma_i \in \ddot{H}_\kappa$ and $(1,\sigma) \in \module$. Then, setting $d=s$, $$\sigma_i \cdot ((1,\sigma)+(1,\sigma_i\sigma)) = s((1,\sigma)+(1,\sigma_i\sigma)).$$
\end{corollary}

\begin{proof}
    Suppose $\sigma(i) < \sigma(i+1)$. Then $\sigma_i \cdot (1,\sigma) = s^{-1}(1,\sigma_i\sigma)$ by Lemma \ref{geodesicisotoping}. On the other hand, $(\sigma_i\sigma)(i)> (\sigma_i \sigma)(i+1)$, so $\sigma_i \cdot (1,\sigma_i\sigma) = \hbar (1,\sigma_i\sigma) + s(1,\sigma)$. Adding the terms, expanding $\hbar = s - s^{-1}$, and canceling the $s^{-1}(1,\sigma_i\sigma)$ gives the result.

    The case $\sigma(i) > \sigma(i+1)$ follows immediately by letting $\sigma = \sigma_i \sigma$. 
\end{proof}

\begin{definition}
\label{def:action}
Let $(\bm{a}, \sigma) = (a_1^{n_1}\dotsm a_\kappa^{n_\kappa}, \sigma) \in \Z [ a_1^{\pm 1}, \dots, a_{\kappa}^{\pm 1} ] \times S_{\kappa}$ be an element of $\module$. The {\em enhanced polynomial representation} of the DAHA $\ddot{H}_\kappa$ (with presentation given in Theorem \ref{DAHAbasis}) on $\module$ is defined on generators as follows:

    (1) $x_i \cdot (\bm{a}, \sigma) = (a_i \cdot \bm{a}, \sigma),$

    (2) $\sigma_i \cdot (1, \sigma) = \begin{cases}
     d^{-1}(1, \sigma_i \sigma) & \text{ if }\sigma(i)<\sigma(i+1)\\
     d(1,\sigma_i \sigma) + \hbar (1, \sigma) & \text{ if } \sigma(i)>\sigma(i+1)
    
    \end{cases},$

   (3) $y_1 \cdot (\bm{a},\sigma) = c^{2n_1}\tau_\kappa^{-1} \cdot (\bm{a}_{\tau_\kappa},\tau_\kappa\sigma)$,
   
   where $\tau_\kappa = \sigma_{\kappa-1}\cdots\sigma_1$ and $\bm{a}_{\tau_{\kappa}} = a_\kappa^{n_1}a_1^{n_2}\cdots a_{\kappa-1}^{n_\kappa}$.\\

(2) in Definition \ref{def:action} only defines the action of $\sigma_i$ on an element $(1,\sigma)$, but we can extend this to an action on $(\bm{a},\sigma)$ by using the action of $x_i$ along with the relations of $\ddot{H}_\kappa$. Since $\sigma_i x_i = x_{i+1} \sigma_i^{-1}$ and $\sigma_i - \sigma_i^{-1} = \hbar$, it follows that $\sigma_i x_i = x_{i+1}(\sigma_i - \hbar) = x_{i+1} \sigma_i - \hbar x_{i+1}$. Similarly, $\sigma_i x_{i+1} = x_i \sigma_i + \hbar x_{i+1}$. If $j \neq i,\ i+1$, then $\sigma_i x_j = x_j \sigma_i$. Thus we are able to express any product of $\sigma_i$ and $x_j$ as an expression where all the $x_j$ are in front of the $\sigma_i$.

Since $(\bm{a},\sigma) = (x_1^{n_1} \cdots x_\kappa^{n_\kappa}) \cdot (1,\sigma)$, it follows that $\sigma_i \cdot (\bm{a},\sigma) = (\sigma_i \cdot x_1^{n_1} \cdots x_\kappa^{n_\kappa}) \cdot (1,\sigma)$. 

\begin{claim}
    We can express $\sigma_i \cdot x_1^{n_1} \cdots x_\kappa^{n_\kappa}$ as $$f(x_1,\cdots, x_\kappa)\sigma_i + g(x_1,\cdots,x_\kappa),$$ where $f,g \in \Z[\hbar,x_1^{\pm1},\cdots, x_\kappa^{\pm 1}]$.
\end{claim}

\begin{proof}
    This follows from repeated applications of the relations discussed above. 
\end{proof}

The action of $\sigma_i$ on a general element $(\bm{a},\sigma)$ then follows from the claim above and Definition \ref{def:action}.

Similarly, we can compute $y_i \cdot (\bm{a},\sigma)$ by using the relation $y_i = \sigma_{i-1}y_{i-1}\sigma_{i-1}$ repeatedly to reduce to an expression of transpositions and $y_1$.

    \begin{proposition}
    \label{prop:action}
        The action defined above is the one given by $$\rho_d: BSk_\kappa(T^2, \bm{q},\bm{q}, \ast) \otimes BSk_\kappa(T^2,\bm{q},\bm{\alpha}, \ast)_d \longrightarrow BSk_\kappa(T^2,\bm{q},\bm{\alpha}, \ast)_d, $$
        where $\rho_d([\gamma_1]_{c},[\gamma_2]_{c,d}) = [\gamma_1\gamma_2]_{c,d}.$
    \end{proposition}

    \begin{proof}
       It suffices to verify (1), (2), and (3) in Definition \ref{def:action}. 
       
       (1) is immediate from the definitions.

       (2) follows from Lemma \ref{geodesicisotoping}.

       (3) is easiest seen on the universal cover of $T^2$; see Figure \ref{fig:y1action} below. We slide the endpoint of the strand going from $q_1$ to $\alpha_{\sigma(i)}$ in $y_1\cdot (\bm{a},\sigma)$ down along $\alpha_{\sigma(i)}$. We can homotope the strand until it looks like the right side of Figure \ref{fig:y1action} without creating any crossings with other strands or the marked point. Thus the two braids are the same element in $BSk_\kappa(T^2,\bm{q},\bm{\alpha}, \ast)_d$. Next, we can pull the strand down across the marked point, picking up a factor of $c^2$ for each marked point we cross in this direction. The resulting braid is equivalent to $c^{2n_1}(\tau_\kappa)^{-1} \cdot (\bm{a}_{\tau_\kappa},\tau_\kappa\sigma)$. This gives (3).
    \end{proof}

\begin{figure}[h]

  \centering
  \includegraphics[width=1.0\linewidth]
  {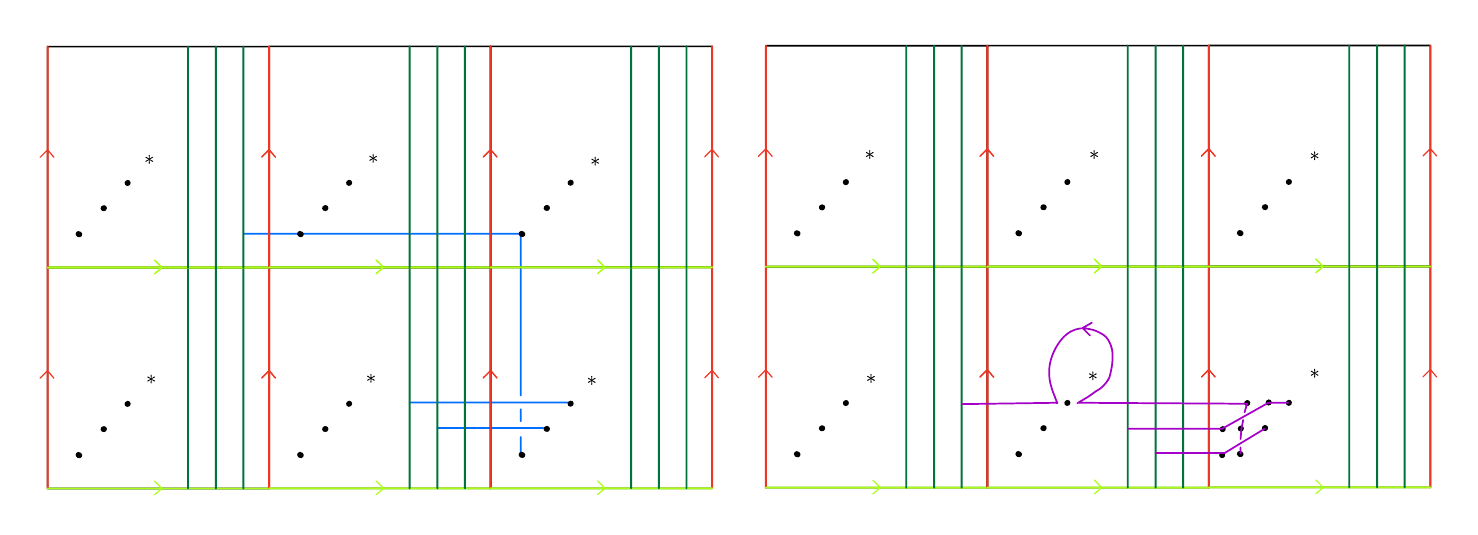}
\caption{The action of $y_1$ on $(a_1,(13))$.} 
\label{fig:y1action}
\end{figure}

\end{definition}

\begin{example}
Suppose $\kappa=2$. We compute $(\sigma_1 \cdot y_1) \cdot (a_1^2 a_2^{-1},\sigma_1)$. First of all, $$y_1 \cdot (a_1^2 a_2^{-1},\sigma_1) = c^4 \sigma_1^{-1} (a_1^{-1}a_2^{2},e).$$ Then $\sigma_1 \cdot c^4 \sigma_1^{-1} (a_1^{-1}a_2^{2},e) = c^4 (a_1^{-1}a_2^{2},e)$.\\
\end{example}

In order to identify our enhanced polynomial representation with the standard polynomial representation from Definition \ref{def:polyrep}, we must first find a way to eliminate the $S_\kappa$ factor of $\module$. The standard polynomial representation acts on $\Z[[s]][c^{\pm 1}][X_1^{\pm 1}, \cdots, X_\kappa^{\pm 1}]$, whereas the generators of $\module$ have permutations $\sigma \in S_\kappa$ associated to them. 

We identify $a_i$ with $X_i$ and substitute $\hbar = s-s^{-1}$ to revert back to the ring $\Z[s^{\pm 1}, c^{\pm 1}]$. We take an average of the permutations which defines a $\Z[[s]][c^{\pm 1}]$-linear map:  $$ \hypertarget{identificationmap}{S}: \Z[[s]][c^{\pm 1}][X_1^{\pm 1}, \cdots, X_\kappa^{\pm 1}] \longrightarrow \module$$
$$ X_1^{n_1}\cdots X_\kappa^{n_\kappa}  \mapsto  \sum_{\sigma \in S_\kappa}(a_1^{n_1} \cdots a_\kappa^{n_\kappa}, \sigma).$$

\begin{remark}
    The $\alpha_i$ are distinct and this is captured by our permutation term $\sigma$ in a generator $(\bm{a},\sigma)$. The permutation-averaging can be thought of as getting rid of this distinctness, resulting in a permutation-free polynomial determined by $\bm{a}$. 
\end{remark}

The following theorem is a more precise version of Theorem \ref{thm:mainaction}.

\begin{theorem}
\label{thm:polyrepiso}
  After setting $d = s$, the standard polynomial representation agrees with the enhanced polynomial representation composed with the permutation-averaging map \hyperlink{identificationmap}{$S$}. 

  More precisely, given an element $h\in \ddot{H}_\kappa$ and an element $f \in \Z[[s]][c^{\pm 1}][X_1^{\pm 1}, \cdots, X_\kappa^{\pm 1}]$, 
  $$S(p(h,f)) = \rho_d (h,S(f)),$$
  \noindent where $p(h,f)$ is the action of the standard polynomial representation and $\rho_d (h,S(f))$ is the action of the enhanced polynomial representation.
\end{theorem}

\begin{proof}
    It suffices to show that the equality holds for the generators as in Definition \ref{def:action}:

        (1) Let $x_i \in \ddot{H}_\kappa$ and $f(X_1,\cdots, X_\kappa) \in \Z[[s]][c^{\pm 1}][X_1^{\pm 1}, \cdots, X_\kappa^{\pm 1}]$. Then 
        \begin{align}
        S(p(x_i,f)) \notag&= S(X_i f) = \sum_{\sigma \in S_\kappa} (a_i f(a_1,\dots, a_\kappa), \sigma)\\ \notag &= \rho_d \bigl( x_i,\sum_{\sigma \in S_\kappa} (f(a_1,\dots, a_\kappa), \sigma)\bigr) = \rho_d(x_i,S(f)).
        \end{align}
        
        (2) Let $\sigma_i \in \ddot{H_\kappa}$ and consider $1 \in \Z[[s]][c^{\pm 1}][X_1^{\pm 1}, \cdots, X_\kappa^{\pm 1}]$.
        
        Let $A_\kappa$ be the alternating group on $\kappa$ elements, i.e. the subgroup of $S_\kappa$ consisting of even permutations. For a transposition $\sigma_i$ and permutation $\sigma \in S_\kappa$, either $\sigma \in A_\kappa$ or $\sigma_i \sigma \in A_\kappa$, so $$\sum_{\sigma \in S_\kappa} (1,\sigma) = \sum_{\rho \in A_\kappa} ((1,\rho) + (1,\sigma_i\rho)).$$ 
        
        \noindent Then, using Corollary \ref{thm:averaging},
        \begin{align}
        S(p(\sigma_i,1)) \notag &= S(s) = \sum_{\sigma \in S_\kappa} s (1, \sigma) =  \sum_{\rho \in A_\kappa} s ((1,\rho) + (1,\sigma_i\rho))\\ \notag & = \sum_{\rho \in A_\kappa} \rho_d \bigl( \sigma_i,(1,\rho) + (1,\sigma_i\rho)\bigr) = \rho_d\bigl( \sigma_i ,\sum_{\sigma \in S_\kappa} (1,\sigma)\bigr) \\ \notag &= \rho_d(\sigma_i,S(1)).
         \end{align}

        (3)  Let $y_1 \in \ddot{H}_\kappa$ and $f(X_1,\cdots, X_\kappa) = X_1^{n_1} \cdots X_\kappa^{n_\kappa} \in \Z[[s]][c^{\pm 1}][X_1^{\pm 1}, \cdots, X_\kappa^{\pm 1}]$. 
        
        Then $\omega(X_1^{n_1} \cdots X_\kappa^{n_\kappa}) =c^{2n_1} X_\kappa^{n_1}X_1^{n_2}\cdots X_{\kappa-1}^{n_\kappa}$.

        Taking advantage of the first two parts of the proof, we see that 
        \begin{align}
        S(p(y_1,f)) \notag &= S(\tau_\kappa^{-1} \omega(f)) = S(\tau_\kappa^{-1} c^{2n_1} X_\kappa^{n_1}X_1^{n_2}\cdots X_{\kappa-1}^{n_\kappa}) \\\notag &= c^{2n_1} \tau_\kappa^{-1} S(X_\kappa^{n_1}X_1^{n_2}\cdots X_{\kappa-1}^{n_\kappa}) = c^{2n_1} \tau_\kappa^{-1} \sum_{\sigma \in S_\kappa} (a_\kappa^{n_1}a_1^{n_2}\cdots a_{\kappa-1}^{n_\kappa},\sigma) \\\notag &= c^{2n_1} \tau_\kappa^{-1} \sum_{\sigma \in S_\kappa} (a_\kappa^{n_1}a_1^{n_2}\cdots a_{\kappa-1}^{n_\kappa},\tau_\kappa \sigma) = y_1 \cdot \sum_{\sigma \in S_\kappa} (f(a_1,\dots,a_\kappa),\sigma)\\ \notag & = \rho_d(y_1, S(f)). 
         \end{align} 
\end{proof}

Theorem \ref{thm:polyrepiso} can be restated as the following corollary:

\begin{corollary}
    Let $W \subset \module$ be the submodule generated over $\Z[[s]][c^{\pm 1}]$ by elements of the form $\sum_{\sigma\in S_\kappa}(\bm{a}, \sigma)$. 
    Then the enhanced polynomial representation has a subrepresentation over $W$ which is isomorphic to the polynomial representation of the double affine Hecke algebra. 
\end{corollary}

\nocite{cherednik}
\nocite{abbondandolo2006floer}
\nocite{abbondandolo2008floer}
\nocite{abbondandolo2010floer}
\nocite{abouzaid2012wrapped}
\nocite{Enomoto2009}
\nocite{morton2021dahas}
\nocite{honda2022higher}

\printbibliography

@article{cherednik,
  title={Double Affine Hecke Algebras and Macdonald’s Conjectures.},
  author={Cherednik, Ivan},
  journal={Annals of Mathematics},
  volume={141},
  number={1},
  pages={191-216},
  year={1995}
}

@misc{colin2020applications,
    title={Applications of higher-dimensional Heegaard Floer homology to contact topology},
    author={Vincent Colin and Ko Honda and Yin Tian},
    year={2020},
    eprint={2006.05701},
    archivePrefix={arXiv},
    primaryClass={math.SG}
}

@article {abouzaid10,
    AUTHOR = {Abouzaid, Mohammed},
     TITLE = {A geometric criterion for generating the {F}ukaya category},
   JOURNAL = {Publ. Math. Inst. Hautes \'{E}tudes Sci.},
  FJOURNAL = {Publications Math\'{e}matiques. Institut de Hautes \'{E}tudes
              Scientifiques},
    NUMBER = {112},
      YEAR = {2010},
     PAGES = {191--240},
      ISSN = {0073-8301,1618-1913},
   MRCLASS = {53D37},
  MRNUMBER = {2737980},
MRREVIEWER = {Timothy\ Perutz},
       DOI = {10.1007/s10240-010-0028-5},
       URL = {https://doi.org/10.1007/s10240-010-0028-5},
}

@article{lipshitz2006cylindrical,
    AUTHOR = {Lipshitz, Robert},
     TITLE = {A cylindrical reformulation of {H}eegaard {F}loer homology},
   JOURNAL = {Geom. Topol.},
  FJOURNAL = {Geometry and Topology},
    VOLUME = {10},
      YEAR = {2006},
     PAGES = {955--1096},
      ISSN = {1465-3060},
   MRCLASS = {57R58 (57M27)},
  MRNUMBER = {2240908},
MRREVIEWER = {Stanislav Jabuka},
       DOI = {10.2140/gt.2006.10.955},
       URL = {https://doi.org/10.2140/gt.2006.10.955},
}

@article {ozsvath2004holomorphic,
    AUTHOR = {Ozsv\'{a}th, Peter and Szab\'{o}, Zolt\'{a}n},
     TITLE = {Holomorphic disks and topological invariants for closed
              three-manifolds},
   JOURNAL = {Ann. of Math. (2)},
  FJOURNAL = {Annals of Mathematics. Second Series},
    VOLUME = {159},
      YEAR = {2004},
    NUMBER = {3},
     PAGES = {1027--1158},
      ISSN = {0003-486X},
   MRCLASS = {57M27 (32Q65 57R58)},
  MRNUMBER = {2113019},
MRREVIEWER = {Thomas E. Mark},
       DOI = {10.4007/annals.2004.159.1027},
       URL = {https://doi.org/10.4007/annals.2004.159.1027},
}

@article {abbondandolo2010floer,
    AUTHOR = {Abbondandolo, Alberto and Schwarz, Matthias},
     TITLE = {Floer homology of cotangent bundles and the loop product},
   JOURNAL = {Geom. Topol.},
  FJOURNAL = {Geometry \& Topology},
    VOLUME = {14},
      YEAR = {2010},
    NUMBER = {3},
     PAGES = {1569--1722},
      ISSN = {1465-3060},
   MRCLASS = {53D40 (55N45 55P50 57R58)},
  MRNUMBER = {2679580},
MRREVIEWER = {Janko Latschev},
       DOI = {10.2140/gt.2010.14.1569},
       URL = {https://doi.org/10.2140/gt.2010.14.1569},
}

@article {abbondandolo2008floer,
    AUTHOR = {Abbondandolo, Alberto and Portaluri, Alessandro and Schwarz, Matthias},
     TITLE = {The homology of path spaces and Floer homology with conormal boundary conditions},
   JOURNAL = {J. fixed point theory appl.},
  FJOURNAL = {Geometry \& Topology},
    VOLUME = {4},
      YEAR = {2008},
    NUMBER = {},
     PAGES = {263-293},
      ISSN = {1465-3060},
   MRCLASS = {53D40 (55N45 55P50 57R58)},
  MRNUMBER = {2679580},
MRREVIEWER = {Janko Latschev},
       DOI = {10.1007/s11784-008-0097-y},
       URL = {https://doi.org/10.1007/s11784-008-0097-y},
}

@article {abbondandolo2006floer,
    AUTHOR = {Abbondandolo, Alberto and Schwarz, Matthias},
     TITLE = {On the {F}loer homology of cotangent bundles},
   JOURNAL = {Comm. Pure Appl. Math.},
  FJOURNAL = {Communications on Pure and Applied Mathematics},
    VOLUME = {59},
      YEAR = {2006},
    NUMBER = {2},
     PAGES = {254--316},
      ISSN = {0010-3640},
   MRCLASS = {53D40 (57R58)},
  MRNUMBER = {2190223},
MRREVIEWER = {Michael J. Usher},
       DOI = {10.1002/cpa.20090},
       URL = {https://doi.org/10.1002/cpa.20090},
}

@article {abouzaid2012wrapped,
    AUTHOR = {Abouzaid, Mohammed},
     TITLE = {On the wrapped {F}ukaya category and based loops},
   JOURNAL = {J. Symplectic Geom.},
  FJOURNAL = {The Journal of Symplectic Geometry},
    VOLUME = {10},
      YEAR = {2012},
    NUMBER = {1},
     PAGES = {27--79},
      ISSN = {1527-5256,1540-2347},
   MRCLASS = {53D37 (53D12)},
  MRNUMBER = {2904032},
MRREVIEWER = {Janko\ Latschev},
       URL = {http://projecteuclid.org/euclid.jsg/1332853049},
}

@article {morton2021dahas,
    AUTHOR = {Morton, Hugh and Samuelson, Peter},
     TITLE = {D{AHA}s and skein theory},
   JOURNAL = {Comm. Math. Phys.},
  FJOURNAL = {Communications in Mathematical Physics},
    VOLUME = {385},
      YEAR = {2021},
    NUMBER = {3},
     PAGES = {1655--1693},
      ISSN = {0010-3616},
   MRCLASS = {57K14 (17B37)},
  MRNUMBER = {4283999},
       DOI = {10.1007/s00220-021-04052-8},
       URL = {https://doi.org/10.1007/s00220-021-04052-8},
}

@article{honda2022higher,
    title={Higher-dimensional Heegaard Floer homology and Hecke algebras},
    author={Honda, Ko and Tian, Yin and Yuan, Tianyu},
    journal={J. Eur. Math. Soc.},
    year={to appear},
    eprint={2202.05593},
    archivePrefix={arXiv},
    primaryClass={math.SG}
}

@article{Enomoto2009,
author = {Naoya Enomoto},
title = {{Composition factors of polynomial representation of DAHA and q-decomposition numbers}},
volume = {49},
journal = {Journal of Mathematics of Kyoto University},
number = {3},
publisher = {Duke University Press},
pages = {441 -- 473},
year = {2009},
doi = {10.1215/kjm/1260975035},
URL = {https://doi.org/10.1215/kjm/1260975035}
}

\end{document}